\documentclass[10pt]{amsart}

\usepackage{amssymb}
\usepackage{amsthm}
\usepackage{amsmath}

\usepackage[foot]{amsaddr}

\usepackage{hyperref}

\theoremstyle{plain}
\newtheorem{proposition}{Proposition}[section]
\newtheorem{theorem}[proposition]{Theorem}
\newtheorem{lemma}[proposition]{Lemma}
\newtheorem{corollary}[proposition]{Corollary}

\theoremstyle{definition}
\newtheorem{example}[proposition]{Example}
\newtheorem{definition}[proposition]{Definition}
\newtheorem{observation}[proposition]{Observation}
\theoremstyle{remark}
\newtheorem{remark}[proposition]{Remark}

\newtheorem{question}{Question}

\DeclareMathOperator{\Aff}{Aff}
\DeclareMathOperator{\Aut}{Aut}

\DeclareMathOperator{\Real}{Re}
\DeclareMathOperator{\Imaginary}{Im}

\DeclareMathOperator{\GL}{GL}
\DeclareMathOperator{\Id}{Id}

\DeclareMathOperator{\Hol}{Hol}

\DeclareMathOperator{\arcosh}{arcosh}

\DeclareMathOperator{\Span}{Span}

\DeclareMathOperator{\Cc}{\mathcal{C}}

\DeclareMathOperator{\Gc}{\mathcal{G}}
\DeclareMathOperator{\Hc}{\mathcal{H}}

\DeclareMathOperator{\Kc}{\mathcal{K}}

\DeclareMathOperator{\Oc}{\mathcal{O}}
\DeclareMathOperator{\Pc}{\mathcal{P}}

\DeclareMathOperator{\Sc}{\mathcal{S}}

\DeclareMathOperator{\Bb}{\mathbb{B}}
\DeclareMathOperator{\Cb}{\mathbb{C}}
\DeclareMathOperator{\Db}{\mathbb{D}}

\DeclareMathOperator{\Kb}{\mathbb{K}}

\DeclareMathOperator{\Nb}{\mathbb{N}}

\DeclareMathOperator{\Rb}{\mathbb{R}}
\DeclareMathOperator{\Xb}{\mathbb{X}}

\newcommand{\abs}[1]{\left|#1\right|}

\newcommand{\norm}[1]{\left\|#1\right\|}

\newcommand{\vect}[1]{\mathbf{#1}}

\begin{document}

\title[Obstructions to biholomorphisms]{Characterizing strong pseudoconvexity, obstructions to biholomorphisms, and Lyapunov exponents}
\author{Andrew Zimmer}\address{Department of Mathematics, College of William and Mary, Williamsburg, VA 23185}
\email{amzimmer@wm.edu}
\date{\today}
\subjclass[2010]{}

\begin{abstract} In this paper we consider the following question: For bounded domains with smooth boundary, can strong pseudoconvexity be characterized in terms of the intrinsic complex geometry of the domain? Our approach to answering this question is based on understanding the dynamical behavior of real geodesics in the Kobayashi metric and allows us to prove a number of results for domains with low regularity.  For instance, we show that for convex domains with $C^{2,\epsilon}$ boundary strong pseudoconvexity can be characterized in terms of the behavior of the squeezing function near the boundary, the behavior of the holomorphic sectional curvature of the Bergman metric near the boundary, or any other reasonable measure of the complex geometry near the boundary. The first characterization gives a partial answer to a question of Forn{\ae}ss and Wold. As an application of these characterizations, we show that a convex domain with $C^{2,\epsilon}$ boundary which is biholomorphic to a strongly pseudoconvex domain is also strongly pseudoconvex.
\end{abstract}

\maketitle

\section{Introduction}

A domain in $\Cb^d$ with $C^2$ boundary is called strongly pseudoconvex if the Levi form of the boundary is positive definite. The Levi form is  extrinsic and in this paper we study the following question: 

\begin{question}\label{question:main}
For domains with $C^2$ boundary, can strong pseudoconvexity be characterized in terms of the intrinsic complex geometry of the domain?
\end{question}

Although strongly pseudoconvex domains form one of the most important classes of domains in several complex variables, it does not appear that Question~\ref{question:main} has been extensively studied. The only general results we know of are due to Bland~\cite{B1985, B1989}, who studies compactifications of complete simply connected non-positively curved K{\"a}hler manifolds whose curvature tensor approaches the curvature tensor of complex hyperbolic space in a controlled way. Under these conditions, Bland proves that the geodesic compactification has a natural CR-structure which is strongly pseudoconvex and uses this to construct bounded holomorphic functions. 

In this paper we will consider only domains in $\Cb^d$, but will avoid needing to control how fast the geometry of the domain approaches the geometry of complex hyperbolic space. We will also focus on the case of convex domains. Convexity is a strong geometric assumption, but in relation to Bland's results can be seen as a non-positive curvature condition. By assuming convexity we are also able to prove results about unbounded domains and domains whose boundary has low regularity.

Our approach to studying Question~\ref{question:main} is based on understanding the behavior of the real geodesics in the Kobayashi metric. Let $\Bb_d \subset \Cb^d$ denote the open unit ball and $K_{\Bb_d}$ denote the Kobayashi distance on $\Bb_d$. Then geodesics in $(\Bb_d, K_{\Bb_d})$ have the following properties:
\begin{enumerate}
\item if $\gamma_1, \gamma_2 : \Rb_{\geq 0} \rightarrow \Bb_d$ are geodesics and $\liminf_{s,t \rightarrow \infty} K_{\Bb_d}(\gamma_1(s), \gamma_2(t)) < \infty$, then there exists $T \in \Rb$ such that $\lim_{t \rightarrow \infty} K_{\Bb_d}(\gamma_1(t), \gamma_2(t+T)) =0$ and
\item if $\gamma_1, \gamma_2 : \Rb_{\geq 0} \rightarrow \Bb_d$ are geodesics and $\lim_{t \rightarrow \infty} K_{\Bb_d}(\gamma_1(t), \gamma_2(t)) =0 $, then
\begin{align*}
\lim_{t \rightarrow \infty} \frac{1}{t} \log K_{\Bb_d}(\gamma_1(t), \gamma_2(t)) = -2
\end{align*}
if  $\gamma_1, \gamma_2$ are contained in the same complex geodesic and 
\begin{align*}
\lim_{t \rightarrow \infty} \frac{1}{t} \log K_{\Bb_d}(\gamma_1(t), \gamma_2(t)) = -1
\end{align*}
otherwise. 
\end{enumerate}
The numbers $ \pm 2$, $ \pm 1$ are exactly the Lyapunov exponents of the geodesic flow on complex hyperbolic space. In Section~\ref{sec:lyapunov}, we will establish, for certain types of convex domains, a relationship between the ``Lyapunov exponents of the geodesic flow'' and the shape of the boundary. This relationship is fundamental in all the results of this paper.

\subsection{Domains biholomorphic to strongly pseudoconvex domains} One of our motivations for studying Question~\ref{question:main} is the following question of Forn{\ae}ss and Wold. 

\begin{question}(Forn{\ae}ss and Wold~\cite[Question 4.5]{FW2017}) Suppose $\Omega \subset \Cb^d$ is a bounded domain with $C^2$ boundary and $\Omega$ is biholomorphic to the unit ball in $\Cb^d$. Is $\Omega$ strongly pseudoconvex?
\end{question}

One can also ask the following more general question:

\begin{question}\label{question:main2} Suppose $\Omega_1, \Omega_2 \subset \Cb^d$ are bounded domains with $C^2$ boundary, $\Omega_1$ is strongly pseudoconvex, and $\Omega_2$ is biholomorphic to $\Omega_1$. Is $\Omega_2$ also strongly pseudoconvex?
\end{question}

When $\Omega_1$ and $\Omega_2$ both have $C^\infty$ boundary, Bell~\cite{B1981} answered the above question in the affirmative using deep analytic methods, namely condition (R) and Kohn's subelliptic estimates in weighted $L^2$-spaces. It does not appear that Bell's analytic approach can be used in the $C^2$ regularity case. 

Using the dynamical approach described above, we will establish the following partial answer to Question~\ref{question:main2}.

\begin{theorem}\label{thm:C2} Suppose $\Omega \subset \Cb^d$ is a bounded strongly pseudoconvex domain with $C^2$ boundary and $\Cc \subset \Cb^d$ is a convex domain biholomorphic to $\Omega$. If $\Cc$ has $C^{2,\alpha}$ boundary for some $\alpha > 0$, then every $x \in \partial\Cc$ is a strongly pseudoconvex point of $\partial \Cc$. 
\end{theorem}
 
\begin{remark} Theorem~\ref{thm:C2} makes no assumptions about the boundedness of $\Cc$. \end{remark}

The dynamical approach also allows us to prove a theorem for convex domains with only $C^1$ boundary, but we need to introduce some additional notation.

\begin{definition}
For a domain $\Omega \subset \Cb^d$, a point $z \in \Omega$, and a non-zero vector $v \in \Cb^d$ define 
\begin{align*}
\delta_\Omega(z) = \inf\{ \norm{z-w} : w \in \partial \Omega \}
\end{align*}
and 
\begin{align*}
\delta_\Omega(z;v) = \inf\{ \norm{z-w} : w \in \partial \Omega \cap (z + \Cb \cdot v) \}.
\end{align*}
\end{definition}

We will then prove the following. 

\begin{theorem}\label{thm:C1}(see Section~\ref{sec:C1})
Suppose $\Omega \subset \Cb^d$ is a bounded strongly pseudoconvex domain with $C^2$ boundary and $\Cc \subset \Cb^d$ is a convex domain biholomorphic to $\Omega$. If $\Cc$ has $C^{1}$ boundary, then for every $\epsilon >0$ and $R> 0$ there exists a $C=C(\epsilon, R) \geq 1$ such that 
\begin{align*}
\delta_{\Cc}(z;v) \leq C \delta_{\Cc}(z)^{1/(2+\epsilon)}
\end{align*}
for all $z \in \Cc$ with $\norm{z} \leq R$ and all nonzero $v \in \Cb^d$.
\end{theorem}

\begin{remark} \ \begin{enumerate} \item Suppose $\Omega \subset \Cb^d$ is bounded, convex, and has $C^2$ boundary. Then $\Omega$ is strongly pseudoconvex if and only if there exists a $C \geq 1$ such that 
\begin{align*}
\delta_\Omega(z;v) \leq C\delta_\Omega(z)^{1/2}
\end{align*}
for all $z \in \Omega$ and all nonzero $v \in \Cb^d$. Thus the conclusion of Theorem~\ref{thm:C1} can be interpreted as saying $\Cc$ is ``almost'' strongly pseudoconvex. 
\item By picking $\epsilon < \alpha$, one sees that Theorem~\ref{thm:C2} is a corollary of Theorem~\ref{thm:C1}.
\end{enumerate}
\end{remark}

%We can also prove a theorem for convex domains with no assumptions on the boundary regularity:
%
%
%\begin{theorem}\label{thm:C0}
%Suppose $\Omega \subset \Cb^d$ is a bounded strongly pseudoconvex domain with $C^2$ boundary and $\Cc \subset \Cb^d$ is a convex domain biholomorphic to $\Omega$. Assume $\xi \in \partial \Cc$ and $\vec{n}, \vec{v} \in \Cb^d$ are unit vectors with the following properties:
%\begin{enumerate}
%\item $\Cc \cap ( \xi + \Cb \cdot \vec{n}) \neq \emptyset$, $\xi$ is a $C^1$ point of $\Omega \cap ( \xi + \Cb \cdot \vec{n})$, and $\vec{n}$ is the inward pointing normal vector of $\Omega \cap ( \xi + \Cb \cdot \vec{n})$ at $\xi$.
%\item $\Cc \cap ( \xi + \Cb \cdot \vec{v}) = \emptyset$.
%\end{enumerate}
%Then for any $\epsilon > 0$ there exists $C,r_0 > 0$ (which depend on $ \xi, \vec{n}, \vec{v}, \epsilon$) so that 
%\begin{align*}
%\frac{1}{C} r^{1/(2+\epsilon)} \leq \delta_{\Cc}(\xi+r\vec{n}; \vec{v}) \leq C r^{1/(2-\epsilon)}
%\end{align*}
%for all $r \in (0, r_0)$.
%\end{theorem}

\subsection{The intrinsic complex geometry of a domain}

There are many ways to measure the complex geometry of a domain and in this subsection we describe how certain natural measures provide characterizations of strong pseudoconvexity amongst convex domains with $C^{2,\alpha}$ boundary. As we will describe in Subsection~\ref{subsec:failure}, a recent example of Forn{\ae}ss and Wold~\cite{FW2017} shows that all these characterizations fail for convex domains with $C^2$ boundary. 

\subsubsection{The squeezing function}

One natural intrinsic measure of the complex geometry of a domain is the squeezing function. Given a bounded domain $\Omega \subset \Cb^d$ let $s_\Omega : \Omega \rightarrow (0,1]$ be the \emph{squeezing function on $\Omega$}, that is 
\begin{align*}
s_\Omega(p) = \sup\{ r : & \text{ there exists an one-to-one holomorphic map } \\
& f: \Omega \rightarrow \Bb_d \text{ with } f(p)=0 \text{ and } r\Bb_d \subset f(\Omega) \}.
\end{align*}
Although only recently introduced, the squeezing function has a number of applications, see for instance~\cite{LSY2004,Y2009}. 

Work of Diederich, Forn{\ae}ss, and Wold~\cite[Theorem 1.1]{DFW2014} and Deng, Guan, and Zhang~\cite[Theorem 1.1]{DGZ2016} implies the following theorem. 

\begin{theorem}\cite{DFW2014, DGZ2016}\label{thm:sq_on_str}
If $\Omega \subset \Cb^d$ is a bounded strongly pseudoconvex domain with $C^2$ boundary, then 
\begin{align*}
\lim_{z \rightarrow \partial \Omega} s_\Omega(z) = 1.
\end{align*}
\end{theorem}

 Based on the above theorem, it seems natural to ask if the converse holds.

\begin{question}\label{question:squeezing} (Forn{\ae}ss and Wold~\cite[Question 4.2]{FW2017}) Suppose $\Omega  \subset \Cb^d$ is a bounded pseudoconvex domain with $C^k$ boundary for some $k > 2$. If 
\begin{align*}
\lim_{z \rightarrow \partial \Omega} s_\Omega(z) = 1,
\end{align*}
is $\Omega$ strongly pseudoconvex?
\end{question}

Surprisingly the answer is no when $k=2$: Forn{\ae}ss and Wold~\cite{FW2017} constructed a convex domain with $C^2$ boundary which is not strongly pseudoconvex, but the squeezing function still approaches one on the boundary. However, we will prove that a little bit more regularity is enough for an affirmative answer. 

\begin{theorem}\label{thm:squeezing}(see Subsection~\ref{subsec:squeezing})
For any $d \geq 2$ and $\alpha >0$, there exists some $\epsilon=\epsilon(d,\alpha) > 0$ such that: if $\Omega  \subset \Cb^d$ is a bounded convex domain with $C^{2,\alpha}$ boundary and 
\begin{align*}
s_\Omega(z) \geq 1-\epsilon
\end{align*}
outside a compact subset of $\Omega$, then $\Omega$ is strongly pseudoconvex. 
\end{theorem}

\begin{remark} Using a different argument, we previously gave an affirmative answer to Question~\ref{question:squeezing} for bounded convex domains with $C^\infty$ boundary~\cite{Z2016}. Moreover, Joo and Kim~\cite{JK2016} gave an affirmative answer for bounded finite type domains in $\Cb^2$ with $C^\infty$ boundary.  
\end{remark}

\subsubsection{Holomorphic sectional curvature of the Bergman metric} Another intrinsic measure of the complex geometry of a domain is the curvature of the Bergman metric.

Let $(X,J)$ be a complex manifold with K{\"a}hler metric $g$. If $R$ is the Riemannian curvature tensor of $(X,g)$, then the \emph{holomorphic sectional curvature} $H_g(v)$ of a nonzero vector $v$ is defined to be the sectional curvature of the 2-plane spanned by $v$ and $Jv$, that is 
\begin{align*}
H_g(v) := \frac{ R(v,Jv,Jv,v)}{\norm{v}_g^4}.
\end{align*}

A classical result of Hawley~\cite{H1953} and Igusa~\cite{I1954} says that if $(X,g)$ is a complete simply connected K{\"a}hler manifold with constant negative holomorphic sectional curvature, then $X$ is biholomorphic to the unit ball (also see Chapter IX, Section 7 in~\cite{KN1996}). Moreover, if $b_{\Bb_d}$ is the Bergman metric on the unit ball $\Bb_d \subset \Cb^d$, then $(\Bb_d, b_{\Bb_d})$ has constant holomorphic sectional curvature $-4/(d+1)$. Klembeck proved that the holomorphic sectional curvature of Bergman metric on a strongly pseudoconvex domain approaches $-4/(d+1)$ on the boundary.

\begin{theorem}[Klembeck~\cite{K1978}]\label{thm:klembeck}
Suppose $\Omega \subset \Cb^d$ is a bounded strongly pseudoconvex domain with $C^\infty$ boundary. Then 
\begin{align*}
\lim_{ z \rightarrow \partial \Omega} \max_{v \in T_z \Omega \setminus \{0\}} \abs{H_{b_\Omega}(v) - \frac{-4}{d+1}} = 0
\end{align*}
where $b_\Omega$ is the Bergman metric on $\Omega$. 
\end{theorem}

We will prove the following converse to Klembeck's theorem:

\begin{theorem}\label{thm:bergman}(see Subsection~\ref{subsec:bergman})
For any $d \geq 2$ and $\alpha >0$, there exists some $\epsilon=\epsilon(d,\alpha) > 0$ such that: if $\Omega  \subset \Cb^d$ is a bounded convex domain with $C^{2,\alpha}$ boundary and 
\begin{align*}
\max_{v \in T_z \Omega \setminus \{0\}} \abs{H_{b_\Omega}(v) - \frac{-4}{d+1}} \leq \epsilon
\end{align*}
outside a compact subset of $\Omega$, then $\Omega$ is strongly pseudoconvex. 
\end{theorem}

\subsubsection{K{\"a}hler metrics with controlled geometry} In Subsection~\ref{subsec:controlled_geom} we will introduce families of K{\"a}hler metrics, denoted by $\Gc_M(\Omega)$ for some $M > 1$, on a convex domain $\Omega$ which have controlled geometry relative to the Kobayashi metric. We will also show that there exists some $M_0 > 1$ such that the Bergman metric is always contained in $\Gc_M(\Omega)$ when $M \geq M_0$. Then we will prove the following generalization of Theorem~\ref{thm:bergman}.

\begin{theorem}\label{thm:gen_riem}(see Subsection~\ref{subsec:controlled_geom})
For any $d \geq 2$, $\alpha > 0$, and $M > 1$, there exists some $\epsilon = \epsilon(d,\alpha,M) > 0$ such that: if $\Omega \subset \Cb^d$ is a bounded convex domain with $C^{2,\alpha}$ boundary and there exists a metric $g \in \Gc_M(\Omega)$ with
\begin{align*}
\max_{v,w \in T_z \Omega \setminus \{0\}} \abs{H_{g}(v) - H_g(w)} \leq \epsilon
\end{align*}
outside a compact subset of $\Omega$, then $\Omega$ is strongly pseudoconvex.
\end{theorem}

\subsubsection{Other intrinsic measures of the complex geometry of a domain} Theorem~\ref{thm:squeezing}, Theorem~\ref{thm:bergman}, and Theorem~\ref{thm:gen_riem} are particular cases of more general theorems which we state and prove in Section~\ref{sec:characterizing}. These more general theorems extend Theorem~\ref{thm:squeezing}, Theorem~\ref{thm:bergman}, and Theorem~\ref{thm:gen_riem}  to essentially any intrinsic measure of the complex geometry of a domain. 

\subsection{Some notations}
\begin{enumerate}
\item For $z \in\Cb^d$, let $\norm{z}$  denote the standard Euclidean norm.
\item For a point $z\in \Cb^d$ and $r > 0$, let 
\begin{align*}
\Bb_d(z;r) = \{ w \in \Cb^d : \norm{w-z} < r\}.
\end{align*}
\item $\Db \subset \Cb$ will denote the open unit disk and $\Bb_d := \Bb_d(0;1) \subset \Cb^d$ will denote the open unit ball. 
\item Let
\begin{align*}
\Db_1 = \{ z \in \Cb : \abs{ {\rm Im}(z)} + \abs{ { \rm Re}(z)} < 1\}.
\end{align*}
\item If $\Cc \subset \Cb^d$ is a convex domain with $C^1$ boundary and $\xi \in \partial \Cc$ let 
\begin{align*}
T_{\xi}^{\Cb} \partial \Cc \subset \Cb^d
\end{align*}
denote the complex tangent space of $\partial \Cc$ at $\xi$. Then since $\Cc$ is convex and open 
\begin{align*}
\left( \xi + T_{\xi}^{\Cb} \partial \Cc \right) \cap \Cc = \emptyset.
\end{align*}
\end{enumerate}

\subsection*{Acknowledgments}  I would like to thank the referee for a number of comments and corrections which improved the present work. This material is based upon work supported by the National Science Foundation under grants DMS-1400919 and DMS-1760233.

\section{Lyapunov exponents and the shape of the boundary}\label{sec:lyapunov}

In this section we establish a relationship between the ``Lyapunov exponents of the geodesic flow'' and the shape of the boundary. This relationship allows us to prove the following result. 

\begin{proposition}\label{prop:obstruction}
Suppose $d \geq 2$ and $\Cc \subset \Cb^d$ is a convex domain with the following properties:
\begin{enumerate}
\item $\Cc \cap\Span_{\Cb} \left\{ e_2, \dots, e_d \right\} = \emptyset$, 
\item $\Cc \cap \Cb \cdot e_1 = \{ z e_1 : \Imaginary(z) > 0\}$, and
\item $\Cc$ is biholomorphic to the unit ball.
\end{enumerate}
Then
\begin{align*}
\lim_{r \rightarrow \infty}\frac{1}{r} \log \delta_{\Cc}(ie^{r}e_1; v) = 1/2
\end{align*}
for all $v \in \Span_{\Cb} \left\{ e_2, \dots, e_d \right\}$.
\end{proposition}

\begin{remark} The unit ball is biholomorphic to the convex domain 
\begin{align*}
\Pc_d = \left\{ (z_1, \dots, z_d) \in \Cb^d : \Imaginary(z_1) > \sum_{i=2}^d \abs{z_i}^2 \right\}
\end{align*}
and this domain satisfies:
\begin{enumerate}
\item[(a)] $\Pc_d \cap\Span_{\Cb} \left\{ e_2, \dots, e_d \right\} = \emptyset$, 
\item[(b)] $\Pc_d \cap \Cb \cdot e_1 = \{ z e_1 : \Imaginary(z) > 0\}$, and
\item[(c)]  $\delta_{\Pc_d}(ie^re_1; v) = e^{r/2}$ for all $r \in \Rb$ and $v \in \Span_{\Cb} \left\{ e_2, \dots, e_d \right\}$.
\end{enumerate}
Hence the above proposition states that if a convex domain is biholomorphic to the unit ball and satisfies conditions (a) and (b) above, then the convex domain asymptotically satisfies condition (c). 
\end{remark}

Before starting the proof of Proposition~\ref{prop:obstruction} we will recall some facts about the Kobayashi pseudo-metric on convex domains and geodesics in complex hyperbolic space. 

\subsection{The Kobayashi metric and distance} In this subsection we recall the definition of the Kobayashi pseudo-metric. A more thorough introduction can be found in~\cite{K2005}.

Given a domain $\Omega \subset \Cb^d$ the \emph{(infinitesimal) Kobayashi pseudo-metric} on $\Omega$ is the pseudo-Finsler metric
\begin{align*}
k_{\Omega}(x;v) = \inf \left\{ \abs{\xi} : f \in \Hol(\Delta, \Omega), \ f(0) = x, \ d(f)_0(\xi) = v \right\}.
\end{align*}
Royden~\cite[Proposition 3]{R1971} proved that the Kobayashi pseudo-metric is an upper semicontinuous function on $\Omega \times \Cb^d$. So, if $\sigma:[a,b] \rightarrow \Omega$ is an absolutely continuous curve (as a map $[a,b] \rightarrow \Cb^d$), then the function 
\begin{align*}
t \in [a,b] \rightarrow k_\Omega(\sigma(t); \sigma^\prime(t))
\end{align*}
is integrable and we can define the \emph{length of $\sigma$} to  be
\begin{align*}
\ell_\Omega(\sigma)= \int_a^b k_\Omega(\sigma(t); \sigma^\prime(t)) dt.
\end{align*}
One can then define the \emph{Kobayashi pseudo-distance} to be
\begin{multline*}
 K_\Omega(x,y) = \inf \left\{\ell_\Omega(\sigma) : \sigma\colon[a,b]
 \rightarrow \Omega \text{ is absolutely continuous}, \right. \\
 \left. \text{ with } \sigma(a)=x, \text{ and } \sigma(b)=y\right\}.
\end{multline*}
This definition is equivalent to the standard definition of $K_\Omega$ via analytic chains, see~\cite[Theorem 3.1]{V1989}.

Directly from the definition one obtains the following property of the Kobayashi pseudo-metric:

\begin{proposition} Suppose $\Omega_1 \subset \Cb^{d_1}$ and $\Omega_2 \subset \Cb^{d_2}$ are domains. If $f: \Omega_1 \rightarrow \Omega_2$ is a holomorphic map, then 
\begin{align*}
K_{\Omega_2}(f(z), f(w)) \leq K_{\Omega_1}(z,w)
\end{align*}
for all $z,w \in \Omega_1$. 
\end{proposition}

For a general domain $\Omega$ it is very hard to determine if $(\Omega, K_\Omega)$ is a Cauchy complete metric space, but for convex domains there is a very simple (to state) characterization due to Barth. 

\begin{theorem}[{Barth~\cite[Theorem 1]{B1980}}]\label{thm:barth} Suppose $\Omega \subset \Cb^d$ is a convex domain. Then the following are equivalent:
\begin{enumerate}
\item $\Omega$ does not contain any complex affine lines, 
\item $K_\Omega$ is non-degenerate and hence a distance on $\Omega$, 
\item $K_\Omega$ is a proper Cauchy complete distance on $\Omega$,
\end{enumerate}
\end{theorem}

\begin{remark} To be precise, Theorem 1 in~\cite{B1980} only states that conditions (1) and (2) are equivalent to $K_\Omega$ being a proper distance on $\Omega$. However, for length spaces any proper distance is also Cauchy complete, see for instance Corollary 3.8 in~\cite[Chapter I]{BH1999}.
\end{remark}

\subsection{Basic estimates for the Kobayashi metric}

In this subsection we recall some basic estimates for the Kobayashi metric on convex domains. All these estimates are very well known, but we provide the short proofs for the reader's convenience. 

\begin{lemma}\label{lem:half_plane}
Suppose $\Omega \subset \Cb^d$ is a convex domain, $V \subset \Cb^d$ is a complex affine line, and $V \cap \Omega$ is a half plane in $V$.  Then
\begin{align*}
K_{\Omega}(z_1,z_2) = K_{ V \cap \Omega}(z_1,z_2)
\end{align*}
for all $z_1, z_2 \in V \cap \Omega$. 
\end{lemma}

\begin{proof} By applying an affine transformation we may assume that
\begin{enumerate}
\item $V \cap \Omega = \{ (z,0,\dots,0) : { \rm Im}(z) > 0\}$ and
\item $\Omega \subset \{ (z_1,\dots, z_d) : { \rm Im}(z_1) > 0\}$.
\end{enumerate}

Applying the distance decreasing property of the Kobayashi metric to the inclusion map $V \cap \Omega \hookrightarrow \Omega$ implies that 
\begin{align*}
K_{\Omega}(z_1,z_2) \leq K_{ V \cap \Omega}(z_1,z_2)
\end{align*}
for all $z_1, z_2 \in V \cap \Omega$. 

Let $P: \Cb^d \rightarrow V$ denote the map $P(z_1,\dots, z_d) = (z_1,0,\dots, 0)$. Then $P(\Omega) = \Omega \cap V$ and $P(z) = z$ for $z \in V$. So applying the distance decreasing property of the Kobayashi metric to the projection map $P: \Omega \rightarrow V \cap \Omega$ implies that 
\begin{align*}
K_{ V \cap \Omega}(z_1,z_2)\leq K_{\Omega}(z_1,z_2)
\end{align*}
for all $z_1, z_2 \in V \cap \Omega$. 
\end{proof}

\begin{lemma}\label{lem:hyperplanes}
Suppose $\Omega \subset \Cb^d$ is a convex domain, $H \subset \Cb^d$ is a complex affine hyperplane such that $H \cap \Omega = \emptyset$, and $P:\Cb^d \rightarrow \Cb$ is an affine map with  $P^{-1}(0)=H$. Then for any $z_1, z_2 \in \Omega$ we have 
\begin{align*}
K_{\Omega}(z_1, z_2) \geq \frac{1}{2}\abs{ \log \abs{\frac{P(z_1)}{P(z_2)}}}.
\end{align*}
\end{lemma}

\begin{proof}
Since $\Omega$ is convex there exists a real hyperplane $H_{\Rb}$ such that $H \subset H_{\Rb}$ and $H_{\Rb} \cap \Omega = \emptyset$. By replacing $P$ with $e^{i\theta} P$ for some $\theta \in \Rb$ we can assume that $P(H_{\Rb}) = \Rb$ and 
\begin{align*}
P(\Omega) \subset \Hc:= \{ z \in \Cb : \Imaginary(z) > 0\}.
\end{align*} 
Then 
\begin{align*}
K_\Omega(z_1, z_2) 
&\geq K_{P(\Omega)}(P(z_1),P(z_2)) \geq K_{\Hc}(P(z_1),P(z_2)) \\
& = \frac{1}{2} \arcosh \left( 1 + \frac{\abs{P(z_1)-P(z_2)}^2}{2 \Imaginary(P(z_1))\Imaginary(P(z_2))} \right) \\
& \geq  \frac{1}{2} \arcosh \left( 1  + \frac{ (\abs{P(z_1)}-\abs{P(z_2)})^2}{2\abs{P(z_1)}\abs{P(z_2)}} \right) \\
& = \frac{1}{2} \arcosh \left( \frac{\abs{P(z_1)}}{\abs{P(z_2)}}+\frac{\abs{P(z_2)}}{\abs{P(z_1)}} \right) = \frac{1}{2}\abs{ \log \abs{\frac{P(z_1)}{P(z_2)}}}.
\end{align*}
\end{proof}

Since every point in the boundary of a convex domain is contained in a supporting hyperplane we have the following consequence of Lemma~\ref{lem:hyperplanes}.

\begin{lemma}\label{lem:convex_lower_bd_2}
Suppose $\Omega \subset \Cb^d$ is a convex domain and $x,y \in \Omega$ are distinct. If $L$ is the complex affine line containing $x, y$,  then 
\begin{align*}
\sup_{\xi \in L \setminus L \cap \Omega} \frac{1}{2}\abs{\log \left( \frac{\norm{x-\xi}}{\norm{y-\xi}} \right)} \leq K_{\Omega}(x,y).
 \end{align*}
 \end{lemma}

\subsection{Geodesics in complex hyperbolic space}

Let $\Bb_d \subset \Cb^d$ be the unit ball. Then it is well known that $(\Bb_d, K_{\Bb_d})$ is a standard model of complex hyperbolic $d$-space. In this subsection we describe some basic properties of geodesics in this metric space, but first a definition. 

\begin{definition}
A \emph{complex geodesic} in a domain $\Omega$ is a holomorphic map $\varphi: \Db \rightarrow \Omega$ which satisfies 
\begin{align*}
K_\Omega(\varphi(z), \varphi(w)) = K_{\Db}(z,w)
\end{align*}
for all $z, w\in \Db$. 
\end{definition}

For the unit ball, every real geodesic is contained in a unique complex geodesic.

\begin{proposition}\label{prop:cplx_geod} If $\gamma : \Rb_{\geq 0} \rightarrow \Bb_d$ is a geodesic ray, then there exists a complex geodesic $\varphi : \Db \rightarrow \Bb_d$ such that $\gamma(\Rb_{\geq 0}) \subset \varphi(\Db)$. Moreover, $\varphi$ is unique up to parametrization, that is: if $\varphi_0 : \Db \rightarrow \Bb_d$ is a complex geodesic with  $\gamma(\Rb_{\geq 0}) \subset \varphi_0(\Db)$ then $\varphi_0 = \varphi \circ \phi$ for some $\phi \in \Aut(\Db)$. 
\end{proposition}

In the proof of Proposition~\ref{prop:obstruction} we will use the following fact about the asymptotic behavior of geodesics in complex hyperbolic space. 

\begin{theorem}\label{thm:cplx_hyp} If $\gamma_1,\gamma_2:\Rb_{\geq 0} \rightarrow \Bb_d$ are geodesic rays such that 
\begin{align*}
\liminf_{s,t \rightarrow \infty} K_{\Bb_d}(\gamma_1(s), \gamma_2(t)) < +\infty, 
\end{align*}
then there exists $T \in \Rb$ such that 
\begin{align*}
\lim_{t \rightarrow \infty} K_{\Bb_d}(\gamma_1(t), \gamma_2(t+T)) =0.
\end{align*}
Moreover, if the images of $\gamma_1$ and  $\gamma_2$ are contained in the same complex geodesic, then 
\begin{align*}
\lim_{t \rightarrow \infty} \frac{1}{t} \log K_{\Bb_d}(\gamma_1(t), \gamma_2(t+T)) = -2
\end{align*}
otherwise 
\begin{align*}
\lim_{t \rightarrow \infty} \frac{1}{t} \log K_{\Bb_d}(\gamma_1(t), \gamma_2(t+T)) = -1.
\end{align*}
\end{theorem}

Although this result is well known, we will sketch the proof of Theorem~\ref{thm:cplx_hyp} in the appendix. 

\subsection{The proof of Proposition~\ref{prop:obstruction}} Before starting the proof we state the following observation:

\begin{observation}\label{obs:dumb} Suppose $\Cc \subset \Cb^d$ is an open convex domain. If $x_0 + \Rb_{\geq 0} \cdot v_0 \subset \Cc$ for some $x_0 \in \Cc$ and $v_0 \in \Cb^d$, then $x + \Rb_{\geq 0} \cdot v_0 \subset \Cc$ for every $x \in \Cc$.
\end{observation}

Now for the rest of the subsection, suppose $d \geq 2$ and $\Cc \subset \Cb^d$ is a convex domain with the following properties:
\begin{enumerate}
\item $\Cc \cap \Span_{\Cb} \left\{ e_2, \dots, e_d \right\} = \emptyset$, 
\item $\Cc \cap \Cb \cdot e_1 = \{ z e_1 : \Imaginary(z) > 0\}$, and
\item $\Cc$ is biholomorphic to the unit ball.
\end{enumerate}

By Observation~\ref{obs:dumb} and property (2) above, for every $v \in \Span_{\Cb}\{e_2, \dots, e_d\}$ there exists some $\alpha_v \in \Rb \cup\{\infty\}$ such that 
\begin{equation}\label{eq:slices}
\{ ze_1 + v  : \Imaginary(z) > \alpha_v \} = \Cc \cap \Big(\Cb \cdot e_1+v\Big).
\end{equation}
Since $\Cc \cap \Span_{\Cb} \left\{ e_2, \dots, e_d \right\} = \emptyset$ we have that $\alpha_v \in \Rb_{\geq 0} \cup \{\infty\}$.

Let $\Sc$ be the set of unit vectors in $\Span_{\Cb} \left\{ e_2, \dots, e_d \right\}$. Then fix some $\delta > 0$ such that 
\begin{align*}
ie_1 + 2\delta \Db \cdot v \subset \Cc
\end{align*}
for every $v \in  \Sc$. Let  $\gamma: \Rb_{\geq 0} \rightarrow \Cb^d$ be the curve given by
\begin{align*}
\gamma(t) = e^{2t} i e_1
\end{align*}
and for $v \in \Sc$ let  $\gamma_v : \Rb_{\geq 0} \rightarrow \Cb^d$ be the curve given by
\begin{align*}
\gamma_v(t) = \delta v + (\alpha_{\delta v} + e^{2t})i e_1
\end{align*}
By Lemma~\ref{lem:half_plane} these curves are geodesic rays in $(\Cc, K_{\Cc})$. \newline

\noindent \textbf{Claim:} For every $v \in \Sc$, 
\begin{align*}
\lim_{t \rightarrow \infty}  K_{\Cc}(\gamma(t), \gamma_v(t)) =0.
\end{align*}

\begin{proof}[Proof of Claim:]
For $t$ large let
\begin{align*}
s_{t,v} = t + \frac{1}{2} \log \left( 1 - \frac{\alpha_{\delta v}}{e^{2t}} \right).
\end{align*}
Then $\gamma_v(s_{t,v}) = \delta v +e^{2t}i e_1$ and 
\begin{align*}
K_{\Cc}(\gamma_v(t), \gamma_v(s_{t,v})) =\frac{1}{2} \abs{\log \left( 1 - \frac{\alpha_{\delta v}}{e^{2t}} \right)}.
\end{align*}
Since $ie_1 + 2\delta \Db \cdot v \subset \Cc$, the equality in~\eqref{eq:slices} implies that 
\begin{align*}
i r e_1 + 2\delta \Db \cdot v \subset \Cc
\end{align*}
for all $ r \geq 1$. Hence 
\begin{align*}
\limsup_{t \rightarrow \infty}  & K_{\Cc}(\gamma(t), \gamma_v(t)) \leq \limsup_{t \rightarrow \infty}  \Big( K_{\Cc}(\gamma(t), \gamma_v(s_{t,v}))+ K_{\Cc}(\gamma_v(s_{t,v}), \gamma_v(t)) \Big)\\
& = \limsup_{t \rightarrow \infty}  K_{\Cc}(\gamma(t), \gamma_v(s_{t,v})) \leq 
\limsup_{t \rightarrow \infty} K_{2 \delta \Db}(0,\delta) < \infty.
\end{align*}
Thus by Theorem~\ref{thm:cplx_hyp} there exists $T_v \in \Rb$ such that
\begin{align*}
\lim_{t \rightarrow \infty}  & K_{\Cc}(\gamma(t), \gamma_v(t+T_v))=0.
\end{align*}
We claim that $T_v=0$. Let $P:\Cb^d \rightarrow \Cb$ be the complex linear map given by $P(z_1, \dots, z_d)=z_1$. Then by Lemma~\ref{lem:hyperplanes}
\begin{align*}
0& =\lim_{t \rightarrow \infty} K_{\Cc}(\gamma(t), \gamma_v(t+T_v)) \geq  \lim_{t \rightarrow \infty} \frac{1}{2} \abs{\log \abs{\frac{P(\gamma(t))}{P(\gamma_v(t+T_v))}}} \\
& =  \lim_{t \rightarrow \infty} \frac{1}{2}\abs{\log \frac{ e^{2t}}{e^{2(t+T_v)}+\alpha_{\delta v}}} =\abs{T_v}
\end{align*}
and so $T_v=0$. 
\end{proof}

By Lemma~\ref{lem:half_plane}, for each $v \in \Sc$ the geodesics $\gamma$ and $\gamma_v$ are contained in different complex geodesics. So by Theorem~\ref{thm:cplx_hyp} for each $v \in \Sc$ we have
\begin{align*}
\lim_{t \rightarrow \infty}  \frac{1}{t} \log K_{\Cc}(\gamma(t), \gamma_v(t)) = -1.
\end{align*}
Moreover
\begin{align*}
\abs{ K_{\Cc}(\gamma(t), \gamma_v(t)) - K_{\Cc}(\gamma(t), \gamma_v(s_{t,v}))}&  \leq K_{\Cc}(\gamma_v(t), \gamma_v(s_{t,v}))= \frac{1}{2} \abs{\log \left( 1 - \frac{\alpha_{\delta v}}{e^{2t}} \right)} \\
 =  \frac{\alpha_{\delta v}}{2} e^{-2t} + {\rm O }\left(e^{-4t} \right).  &
\end{align*}
So we also have
\begin{align*}
\lim_{t \rightarrow \infty}  \frac{1}{t} \log K_{\Cc}(\gamma(t), \gamma_v(s_{t,v})) = -1.
\end{align*}

\noindent \textbf{Claim:} For every $v \in \Sc$, 
\begin{align*}
\limsup_{t \rightarrow \infty} \frac{1}{t} \log \delta_{\Cc}(e^{t}ie_1; v)  \leq 1/2
\end{align*}

\begin{proof}[Proof of Claim:]
Note that 
\begin{align*}
K_{\Cc}(\gamma(t), \gamma_v(s_{t,v})) \leq K_{\delta_{\Cc}(e^{2t}ie_1; v)\Db}(0, \delta) = K_{\Db}\left(0, \frac{1}{\delta_{\Cc}(e^{2t}ie_1; v)}\delta \right).
\end{align*}
Then since $K_{\Db}$ is locally Lipschitz on $\Db \times \Db$ and $\delta_{\Cc}(e^{2t}ie_1; v) \geq 2 \delta$, there exists some $C \geq 0$ such that 
\begin{align*}
K_{\Cc}(\gamma(t), \gamma_v(s_{t,v})) \leq C\frac{ \delta}{\delta_{\Cc}(e^{2t}ie_1; v)}.
\end{align*}
Hence 
\begin{align*}
-1 = \lim_{t \rightarrow \infty} &  \frac{1}{t} \log K_{\Cc}(\gamma(t), \gamma_v(s_{t,v})) \\
& \leq \liminf_{t \rightarrow \infty}  -\frac{1}{t} \log \delta_{\Cc}(e^{2t}ie_1; v) = - \limsup_{t \rightarrow \infty} \frac{1}{t} \log \delta_{\Cc}(e^{2t}ie_1; v) \\
& = -2  \limsup_{t \rightarrow \infty} \frac{1}{t} \log \delta_{\Cc}(e^{t}ie_1; v)
\end{align*}
\end{proof}

\noindent \textbf{Claim:} For every $v \in \Sc$, 
\begin{align*}
\liminf_{t \rightarrow \infty} \frac{1}{t} \log \delta_{\Cc}(e^{t}ie_1; v)  \geq 1/2.
\end{align*}

\begin{proof}[Proof of Claim:] Fix a sequence $t_n \rightarrow \infty$ such that 
\begin{align*}
\liminf_{t \rightarrow \infty} \frac{1}{t} \log \delta_{\Cc}(e^{t}ie_1; v)  = \lim_{n \rightarrow \infty} \frac{1}{2t_n} \log \delta_{\Cc}(e^{2t_n}ie_1; v).
\end{align*}
Then let $z_n \in \Cb$ be such that $\abs{z_n} = \delta_{\Cc}(e^{2t_n}ie_1; v)$ and $e^{2t_n}ie_1+z_n v \in \partial \Cc$. By passing to a subsequence we can suppose that
\begin{align*}
\frac{z_n}{\abs{z_n}} \rightarrow e^{i \theta}
\end{align*}
for some $\theta \in \Rb$.

Let $v_0 = -e^{i\theta}v$, then by Lemma~\ref{lem:convex_lower_bd_2} we have
\begin{align*}
    K_{\Cc}(\gamma(t_n), \gamma_{v_0}(s_{t_n,v_0})) & \geq \frac{1}{2} \abs{ \log \frac{ \norm{ \gamma(t_n) - (e^{2t_n}ie_1+z_n v)}}{\norm{ \gamma_{v_0}(s_{t_n,v_0})- (e^{2t_n}ie_1+z_n v)}} } =  \frac{1}{2} \abs{ \log \frac{ \abs{z_n} }{\abs{\delta e^{i\theta} +z_n }} }.
 \end{align*}
    Now for $n$ large 
 \begin{align*}
 \abs{\delta e^{i\theta} +z_n } > \abs{z_n} +\delta/2
 \end{align*}
 and so for $n$ large
 \begin{align*}
    K_{\Cc}(\gamma(t_n), \gamma_{v_0}(s_{t_n,v_0}))  \geq \frac{1}{2} \abs{ \log \frac{ \abs{z_n} }{\abs{\delta e^{i\theta} +z_n }} } = \frac{1}{2} \log \frac{\abs{\delta e^{i\theta} +z_n }}{ \abs{z_n} } \geq \frac{1}{2} \log \left( 1 + \frac{\delta}{2\abs{z_n}} \right).
    \end{align*}
    Since 
    \begin{align*}
  \lim_{n \rightarrow \infty}    K_{\Cc}(\gamma(t_n), \gamma_{v_0}(s_{t_n,v_0}))    =0,
  \end{align*}
  the above estimate implies that $\abs{z_n} \rightarrow \infty$, then using the fact that $\log : \Rb_{> 0} \rightarrow \Rb$ is locally bi-Lipschitz there exists some $C > 0$ such that 
\begin{align*}
 K_{\Cc}(\gamma(t_n), \gamma_{v_0}(s_{t_n,v_0}))  \geq \frac{C}{\abs{z_n}} = \frac{C}{ \delta_{\Cc}(e^{2t_n}ie_1; v)}.
\end{align*}
Hence
\begin{align*}
-1 = \lim_{t \rightarrow \infty} &  \frac{1}{t} \log K_{\Cc}(\gamma(t), \gamma_{v_0}(s_{t,v})) \\
& \geq \limsup_{n \rightarrow \infty}  \frac{1}{t_n} \log \frac{C}{ \delta_{\Cc}(e^{2t_n}ie_1; v)} = - \liminf_{n \rightarrow \infty} \frac{1}{t_n} \log \delta_{\Cc}(e^{2t_n}ie_1; v) \\
& = -2\liminf_{t \rightarrow \infty} \frac{1}{t} \log \delta_{\Cc}(e^{t}ie_1; v).
\end{align*}

\end{proof}

\section{The space of convex domains and the action of the affine group}\label{sec:rescaling}

Following work of Frankel~\cite{F1989b, F1991}, in this section we describe some facts about the space of convex domains and the action of the affine group on this space. 

\begin{definition} Let $\Xb_d$ be the set of convex domains in $\Cb^d$ which do not contain a complex affine line and let $\Xb_{d,0}$ be the set of pairs $(\Omega,x)$ where $\Omega \in \Xb_d$ and $x \in \Omega$. \end{definition}

\begin{remark} The motivation for only considering convex domains which do not contain complex affine lines comes from Theorem~\ref{thm:barth}.
\end{remark}

We now describe a natural topology on the sets $\Xb_d$ and $\Xb_{d,0}$. Given two compact sets $A,B \subset \Cb^d$ define the \emph{Hausdorff distance} between them to be
\begin{align*}
d_{H}(A,B) = \max\left\{ \max_{a \in A} \min_{b \in B} \norm{a-b}, \max_{b \in B} \min_{a \in A} \norm{b-a} \right\}.
\end{align*}
The Hausdorff distance is a complete metric on the set of compact subsets of $\Cb^d$. To consider general closed sets, we introduce the \emph{local Hausdorff semi-norms} between two closed sets $A,B \subset \Cb^d$ by defining 
\begin{align*}
d_{H}^{(R)}(A,B) = d_{H}\left(A \cap \overline{\Bb_d(0;R)},B\cap \overline{\Bb_d(0;R)}\right)
\end{align*}
for $R > 0$. Since an open convex set is determined by its closure, we can define a topology on $\Xb_{d}$ and $\Xb_{d,0}$ using these seminorms:
\begin{enumerate}
\item A sequence $\Cc_n \in \Xb_d$ converges to $\Cc \in \Xb_d$ if there exists some $R_0 \geq 0$ such that $d_{H}^{(R)}(\overline{\Cc}_n,\overline{\Cc}) \rightarrow 0$ for all $R \geq R_0$, 
 \item A sequence $(\Cc_n,x_n) \in \Xb_{d,0}$ converges to $(\Cc,x) \in \Xb_{d,0}$ if $\Cc_n$ converges to $\Cc$ in $\Xb_d$ and $x_n$ converges to $x$ in $\Cb^d$. 
 \end{enumerate}

Let $\Aff(\Cb^d)$ be the group of complex affine isomorphisms of $\Cb^d$. Then $\Aff(\Cb^d)$ acts on $\Xb_d$ and $\Xb_{d,0}$. Remarkably, the action of $\Aff(\Cb^d)$ on $\Xb_{d,0}$ is co-compact:

\begin{theorem}[Frankel \cite{F1991}]\label{thm:frankel_compactness}
The group $\Aff(\Cb^d)$ acts co-compactly on $\Xb_{d,0}$, that is there exists a compact set $K \subset \Xb_{d,0}$ such that $\Aff(\Cb^d) \cdot K = \Xb_{d,0}$. 
\end{theorem}

Given some $\Cc$ and a sequence of points $x_n \in \Cc$ the above theorem says that we can find affine maps $A_n \in \Aff(\Cb^d)$ such that $\{A_n (\Cc,x_n) \}_{n \in \Nb}$ is relatively compact in $\Xb_{d,0}$. Hence there exists a subsequence $n_k \rightarrow \infty$ such that $A_{n_k} (\Cc,x_{n_k})$ converges in $\Xb_{d,0}$. Many of the arguments that follow rely on analyzing the geometry of the domains obtained by this ``rescaling'' which leads to the next definition.

\begin{definition} Given some $\Cc \in \Xb_d$ let ${\rm BlowUp}(\Cc)\subset \Xb_d$ denote the set of $\Cc_\infty$ in $\Xb_d$ where there exist a sequence $x_n \in \Cc$, a point $x_\infty \in \Cc_\infty$, and affine maps $A_n \in \Aff(\Cb^d)$ such that 
\begin{enumerate}
\item $x_n \rightarrow \infty$ in $\Cc$ (that is, for every compact subset $K \subset \Cc$ there exists some $N > 0$ such that $x_n \notin K$ for all $n \geq N$),
\item $A_n(\Cc, x_n)$ converges to $(\Cc_\infty, x_\infty)$. 
\end{enumerate}
\end{definition}

For some domains, the set ${\rm BlowUp}(\Cc)$ is very special.

\begin{proposition}\label{prop:blow_up} Suppose that $\Cc \subset \Cb^d$ is a convex domain which is biholomorphic to a bounded strongly pseudoconvex domain with $C^2$ boundary. Then every $\Cc_\infty \in {\rm BlowUp}(\Cc)$ is biholomorphic to the unit ball in $\Cb^d$. 
\end{proposition}

This is a consequence of the Frankel-Pinchuk rescaling method, but we will provide a proof using the squeezing function. 

\begin{proof} Suppose that  $\Cc_\infty \in {\rm BlowUp}(\Cc)$. Then fix a sequence $x_n \in \Cc$ such that $x_n \rightarrow \infty$ in $\Cc$, a point $x_\infty \in \Cc_\infty$, and affine maps $A_n \in \Aff(\Cb^d)$ such that $A_n(\Cc, x_n)$ converges to $(\Cc_\infty, x_\infty)$.

By results of Diederich, Forn{\ae}ss, and Wold~\cite[Theorem 1.1]{DFW2014} and Deng, Guan, and Zhang~\cite[Theorem 1.1]{DGZ2016} (see Theorem~\ref{thm:sq_on_str} above) 
\begin{align*}
\lim_{n \rightarrow \infty} s_{\Cc}(x_n) = 1.
\end{align*}
Now the function $(\Omega, x) \in \Xb_{d,0} \rightarrow s_{\Omega}(x)$ is an upper semicontinuous function on $\Xb_{d,0}$ (see Proposition 7.1 in~\cite{Z2016}). So
\begin{align*}
s_{\Cc_\infty}(x_\infty) \geq \lim_{n \rightarrow \infty} s_{A_n\Cc}(A_nx_n) =\lim_{n \rightarrow \infty} s_{\Cc}(x_n) = 1.
\end{align*}
Then $s_{\Cc_\infty}(x_\infty)=1$ and so $\Cc_\infty$ is biholomorphic to the unit ball in $\Cb^d$ by Theorem 2.1 in~\cite{DGZ2012}.
\end{proof}

We next define a particular compact subset of $\Xb_{d,0}$ whose $\Aff(\Cb^d)$-translates cover $\Xb_{d,0}$. Recall that 
\begin{align*}
\Db_1 = \{ z \in \Cb : \abs{ {\rm Im}(z)} + \abs{ { \rm Re}(z)} < 1\}.
\end{align*}
For $1 \leq i \leq d$ consider the complex $(d-i)$-dimensional affine plane
\begin{align*}
Z_i =  e_i + \Span_{\Cb}\{e_{i+1}, \dots, e_d\}.
\end{align*}

\begin{definition}\label{def:compact_set}
Let $\Kb_d \subset \Xb_{d}$ be the set of convex domains $\Omega$ such that:
\begin{enumerate} 
\item $\Db_1 e_i \subset \Omega$ for each $1 \leq i \leq d$, 
\item $Z_i \cap \Omega=\emptyset$ for each $1 \leq i \leq d$. 
\end{enumerate}
Also let $\Kb_{d,0} = \{ (\Omega, 0) : \Omega \in \Kb\}$.
\end{definition} 

\begin{theorem}\label{thm:weak_fund_domain}\cite[Theorem 2.5]{Z2016} With the notation above: $\Kb_{d,0}$ is a compact subset of $\Xb_{d,0}$ and $\Aff(\Cb^d) \cdot \Kb_{d,0} = \Xb_{d,0}$.
\end{theorem}

\begin{remark} In~\cite{Z2016} the set $\Kb_d \subset \Xb_d$ was slightly different: in particular one had the requirement that 
\begin{align*}
\Db e_i \subset \Omega \text{ for each } 1 \leq i \leq d
\end{align*}
instead of 
\begin{align*}
\Db_1 e_i \subset \Omega \text{ for each } 1 \leq i \leq d.
\end{align*}
However, the proof is identical.  
\end{remark}

We end this section with a technical result which will allow us to reduce calculations to the two dimensional case. 

\begin{proposition}\label{prop:two_diml_slice}
Suppose $\Cc \in \Xb_{d}$ is a convex domain such that:
\begin{enumerate}
\item $\Cc \cap \left(e_1 + \Span_{\Cb}\{e_2, \dots, e_d\}\right) = \emptyset$ and 
\item $\Cc \cap \Span_{\Cb} \{ e_1, e_2\} \in \Kb_2$,
\end{enumerate}
then there exists $A \in \GL_d(\Cb)$ such that $A|_{\Span_{\Cb}\{e_1, e_2\}} = \Id_{\Span_{\Cb}\{e_1, e_2\}}$ and $A\Cc \in \Kb_d$. 
\end{proposition}

\begin{proof} We will select points $\xi_1, \dots, \xi_d \in \partial \Cc$ and subspaces $H_1, \dots, H_d \subset \Cb^d$ as follows. First let $\xi_1 = e_1$ and $H_1 = \Span_{\Cb}\{e_2, \dots, e_d\}$. Then let $\xi_2 = e_2$ and let $H_2$ be a $(d-2)$-dimensional complex subspace such that $(e_2 + H_2) \cap \Cc = \emptyset$ and 
\begin{align*}
H_2 \subset H_1 =  \Span_{\Cb}\{e_2, \dots, e_d\}.
\end{align*}
 Since $\Span_{\Cb}\{e_2, \dots, e_d\} \cap \Cc$ is convex and $e_2 \in \partial \Cc$, such a subspace exists. Now supposing that $\xi_1, \dots, \xi_{k-1}$ and $H_1, \dots, H_{k-1}$ have already been selected, we pick $\xi_k$ and $H_k$ as follows: let $\xi_k$ be a point in $H_{k-1} \cap \partial \Cc$ closest to $0$ and let $H_k$ be a $(d-k)$-dimensional complex subspace such that $H_k \subset H_{k-1}$ and $(\xi_k + H_k) \cap \Cc = \emptyset$. Since $H_{k-1} \cap \Cc$ is convex and $\xi_k \in \partial (H_{k-1} \cap \Cc)$, such a subspace exists.

Notice that
\begin{enumerate}
\item $\Cb \cdot \xi_k + H_k = H_{k-1}$ for $k \geq 2$, 
\item $H_{k} = \Span_{\Cb}\{ \xi_{k+1}, \dots, \xi_d\}$ for $k \geq 1$, and
\item $\Span_{\Cb}\{ \xi_1,  \dots, \xi_{d}\} = \Cb^d$.
\end{enumerate}

Now let $A \in \GL_d(\Cb)$ be the complex linear map with $A(\xi_i)=e_i$ for $1 \leq i \leq d$. Since $\xi_1, \dots, \xi_d$ is a basis of $\Cb^d$, the linear map $A$ is well defined. Since $\xi_1=e_1$ and $\xi_2 = e_2$ we see that $A|_{\Span_{\Cb}\{e_1, e_2\}} = \Id_{\Span_{\Cb}\{e_1, e_2\}}$.

We now claim that $A\Cc \in \Kb_d$. Since $A\Cc \cap \Span_{\Cb}\{e_1,e_2\} \in \Kb_2$ we have 
\begin{align*}
\Db_1 \cdot e_i \subset A\Cc \text{ for } i=1,2
\end{align*}
and by construction  
\begin{align*}
\Db \cdot e_i \subset A\Cc \text{ for } i=3,\dots, d.
\end{align*}
So 
\begin{align*}
\Db_1 \cdot e_i \subset A\Cc \text{ for } i=1,\dots, d.
\end{align*}
Since $A(\xi_k)=e_k$ and $H_{k} = \Span_{\Cb}\{ \xi_{k+1}, \dots, \xi_d\}$ we have
\begin{align*}
A\Cc \cap Z_k & = A \left( \Cc \cap A^{-1}Z_k \right) = A \left( \Cc \cap ( \xi_k + \Span_{\Cb}\{ \xi_{k+1}, \dots, \xi_d\})\right) \\
& = A\left( \Cc \cap (\xi_k+H_k)\right) =  \emptyset.
\end{align*}
So $A\Cc \in \Kb_d$.
\end{proof}

\section{The proof of Theorem~\ref{thm:C1}}\label{sec:C1}

In this section we will prove Theorem~\ref{thm:C1} which we begin by recalling.

\begin{theorem}
Suppose $\Omega \subset \Cb^d$ is a bounded strongly pseudoconvex domain with $C^2$ boundary and $\Cc \subset \Cb^d$ is a convex domain biholomorphic to $\Omega$. If $\Cc$ has $C^{1}$ boundary, then for every $\epsilon >0$ and $R> 0$ there exists a $C=C(\epsilon, R) \geq 1$ such that 
\begin{align*}
\delta_{\Cc}(z;v) \leq C \delta_{\Cc}(z)^{1/(2+\epsilon)}
\end{align*}
for all $z \in \Cc$ with $\norm{z} \leq R$ and all nonzero $v \in \Cb^d$.
\end{theorem}

For the rest of the section, fix a convex domain $\Cc \subset \Cb^d$ satisfying the conditions of the theorem. Then fix some $\epsilon >0$ and $R> 0$. 

For $z \in \Cc$ let $P_z$ be the set of points in $\partial \Cc$ which are closest to $z$. Then pick $R^\prime \geq R$ such that
\begin{align*}
P_z \subset \Bb_d(0;R^\prime)
\end{align*}
for all $z \in \overline{\Bb_d(0;R)} \cap \Cc$. Next let $\Kc = \overline{\Bb_d(0;R^\prime)} \cap \partial \Cc$. For $\xi \in \partial \Cc$ let $\vect{n}(\xi)$ be the inward pointing unit normal vector of $\Cc$ at $\xi$. Finally fix $\delta \in (0,1)$ such that 
\begin{align*}
\xi + r \vect{n}(\xi) \in \Cc
\end{align*}
for all $\xi \in \Kc$ and $r \in (0,\delta]$. As before let
\begin{align*}
\Db_1 = \{ z \in \Cb : \abs{ {\rm Im}(z)} + \abs{ { \rm Re}(z)} < 1\}.
\end{align*}
Since $\partial \Cc$ is $C^1$, by shrinking $\delta > 0$ if necessary, we can assume that 
\begin{align*}
\xi + \delta \vect{n}(\xi) +\delta \Db_1 \cdot\vect{n}(\xi) \subset \Cc
\end{align*}
for all $\xi \in \Kc$. Then 
\begin{align*}
\xi +r \vect{n}(\xi) +r \Db_1 \cdot\vect{n}(\xi) \subset \xi + \delta \vect{n}(\xi) +\delta \Db_1 \cdot\vect{n}(\xi) \subset \Cc
\end{align*}
for all $\xi \in \Kc$ and $r \in (0,\delta]$.

We begin by showing that the desired estimate holds for tangential directions. 

\begin{lemma} With the notation above, there exists $C_0 > 1$ such that 
\begin{align*}
\delta_{\Cc}(\xi + r \vect{n}(\xi) ;v)\leq C_0 r^{1/(2 +\epsilon)}
\end{align*}
for all $\xi \in \Kc$, $r \in (0,\delta]$, and nonzero $v \in T_{\xi}^{\Cb} \partial \Cc$.
\end{lemma}

\begin{proof}
Suppose not, then there exist $\xi_n \in \Kc$, $r_n \in (0,\delta]$, unit vectors $v_n \in T_{\xi_n}^{\Cb} \partial \Cc$, and $C_n >0$ such that $C_n \rightarrow \infty$ and 
\begin{align*}
\delta_{\Cc}(\xi_n + r_n \vect{n}(\xi_n) ;v_n) = C_n r_n^{1/(2 +\epsilon)}.
\end{align*}
By increasing $r_n$ if necessary we can assume in addition that 
\begin{align*}
\delta_{\Cc}(\xi_n + r \vect{n}(\xi_n) ;v_n) \leq C_n r^{1/(2 +\epsilon)}
\end{align*}
for all $r \in [r_n, \delta]$. Since $\Cc$ contains no complex affine lines, we must have $r_n \rightarrow 0$. 

Now for each $n$, let $\tau_n : \Cb^d \rightarrow \Cb^d$ be an affine isometry such that 
\begin{enumerate}
\item $\tau_n(\xi_n) = 0$, 
\item $\tau_n(\xi_n+\vect{n}(\xi_n)) = ie_1$, 
\item $\tau_n\left(\xi_n + v_n\right) = e_2$.
\end{enumerate}
Conditions (1) and (2) imply that 
\begin{align*}
T_0\tau_n(\partial \Cc) = \{ (z_1, \dots, z_n) \in \Cb^n : \Imaginary(z_1) = 0\}
\end{align*}
and
\begin{align*}
\tau_n(  \Cc) \subset  \{ (z_1, \dots, z_n) \in \Cb^n : \Imaginary(z_1) > 0\}.
\end{align*}

Condition (3) implies that 
\begin{align*}
\delta_{\tau_n(\Cc)}(r_nie_1;e_2) = C_n r_n^{1/(2+\epsilon)}.
\end{align*}
Then pick $z_n \in \Cb$ such that $\abs{z_n} = C_nr_n^{1/(2+\epsilon)}$ and 
\begin{align*}
r_n i e_1 + z_n e_2 \in \partial \tau_n\Cc.
\end{align*}
Then consider the diagonal matrix 
 \begin{align*}
 A_n = \begin{pmatrix} \frac{1}{r_n} & & & & \\  & \frac{1}{z_n} & & & \\ & & 1 & & \\ & & & \ddots & \\ & & & & 1 \end{pmatrix}.
 \end{align*}
Let $\Cc_n = A_n \tau_n(\Cc)$. Since
\begin{align*}
\xi_n + r_n\vect{n}(\xi_n) + r_n\Db_1\cdot \vect{n}(\xi_n) \subset \Cc
\end{align*}
we see that 
\begin{align*}
ie_1 +  \Db_1 \cdot e_1 \subset \Cc_{n}.
\end{align*}
Further, by construction: 
 \begin{enumerate}
 \item $\{ (z_1, \dots, z_n) \in \Cb^n : \Imaginary(z_1) = 0\} \cap \Cc_{n} = \emptyset$,
\item $ie_1 + e_2 \notin \Cc_{n}$, and
\item $ie_1 +  \Db \cdot e_2 \subset \Cc_{n}$.
\end{enumerate}
Hence $\Cc_n \cap \Span_{\Cb} \{ e_1, e_2\} \subset ie_1+\Kb_2$ where $\Kb_2 \subset \Xb_2$ is the subset from Definition~\ref{def:compact_set}. Now by Proposition~\ref{prop:two_diml_slice} there exists an affine map $B_n \in \Aff(\Cb^d)$ such that $B_n|_{\Span_{\Cb}\{e_1,e_2\}}= \Id_{\Span_{\Cb}\{e_1,e_2\}}$ and $ B_n\Cc_n \in ie_1+ \Kb_d$. 

Now since $\Kb_d$ is compact in $\Xb_d$, we can pass to a subsequence such that $B_n \Cc_n$ converges to some $\Cc_\infty$ in $\Xb_d$. Notice that $B_n \Cc_n = B_n A_n \tau_n \Cc$ and 
\begin{align*}
ie_1 = (B_nA_n\tau_n)(\xi_n+r_n\vect{n}(\xi_n)).
\end{align*}
 Since $r_n \rightarrow 0$  and $ie_1 \in \Cc_\infty$ we see that 
\begin{align*}
\Cc_\infty \in \mathrm{BlowUp}(\Cc).
\end{align*}

We next claim that $\Cc_\infty$ satisfies conditions (1), (2), and (3) from Proposition~\ref{prop:obstruction}. By Proposition~\ref{prop:blow_up}, $\Cc_\infty$ is biholomorphic to the unit ball and hence satisfies condition (3). 

Since each $B_n \Cc_n$ is in $ie_1+\Kb_d$, we see that 
\begin{align*}
\{ (z_1, \dots, z_n) \in \Cb^n : \Imaginary(z_1) = 0\} \cap B_n\Cc_{n} = \emptyset
\end{align*}
and so 
\begin{align*}
\{ (z_1, \dots, z_n) \in \Cb^n : \Imaginary(z_1) = 0\} \cap \Cc_{\infty} = \emptyset.
\end{align*}
Hence $\Cc_\infty$ satisfies condition (1). 

For $\eta > 0$ and $r \in (0,\infty]$ let 
\begin{align*}
A(r; \eta) = \{ z \in \Cb : 0 < \abs{z} < r \text{ and } \abs{\Imaginary(z)} < \eta \Real(z) \}.
\end{align*}
Since $\Kc \subset \partial\Cc$ is compact and $\partial \Cc$ is a $C^1$ hypersurface, for any $\eta > 0$ there exists some $r_\eta > 0$ such that 
\begin{align*}
\xi + A(r_\eta; \eta) \cdot \vect{n}(\xi)  \subset \Cc
\end{align*}
for all $\xi \in \Kc$. Then for any $\eta > 0$ we have 
\begin{align*}
A(r_\eta/r_n; \eta)\cdot ie_1 \subset B_n \Cc_n
\end{align*}
and so 
\begin{align*}
A(\infty; \eta)\cdot ie_1 \subset  \Cc_\infty.
\end{align*}
Since $\eta > 0$ was arbitrary and 
\begin{align*}
\Cc_{\infty} \subset \{ (z_1,\dots,z_d) \in \Cb^d :  \Imaginary(z_1) > 0\}
\end{align*} 
we then have
\begin{align*}
\{ ze_1 : \Imaginary(z) > 0\} = \Cc_{\infty} \cap \Cb \cdot e_1.
\end{align*}
Hence $\Cc_\infty$ satisfies condition (2).

However, if $1 \leq r \leq \delta/r_n$, then 
\begin{align*}
\delta_{B_n\Cc_{n}}(rie_1; e_2)& = \delta_{\Cc_{n}}(rie_1; e_2)  = \frac{1}{\abs{z_n}}\delta_{\tau_n(\Cc)}(r_n r ie_1; e_2)  \\
& = \frac{1}{\abs{z_n}}\delta_{\Cc}(\xi+r_n r\vect{n}(\xi_n); v_n)  \leq \frac{1}{\abs{z_n}} C_n (r_nr)^{1/(2+\epsilon)}= r^{1/(2+\epsilon)}.
\end{align*}
So for $1 \leq r$ we have 
\begin{equation*}
\delta_{\Cc_{\infty}}(rie_1; e_2)  \leq  r^{1/(2+\epsilon)}.
\end{equation*}
Which Proposition~\ref{prop:obstruction} says is impossible. So we have a contradiction. 
\end{proof}

We now prove the desired estimate for all directions. 

\begin{lemma} With the notation above, there exists $C\geq1$ such that
\begin{align*}
\delta_{\Cc}(x ;v)\leq C \delta_{\Cc}(x)^{1/(2+\epsilon)}
\end{align*}
for all $x \in \overline{\Bb_d(0;R)} \cap \Cc$ and all nonzero $v \in \Cb^d$.
\end{lemma}

\begin{proof}
Since $\Cc$ does not contain any complex affine lines, there exists $M >0$ such that 
\begin{align*}
\delta_{\Cc}(x; v) \leq M
\end{align*}
for all $x \in \overline{\Bb_d(0;R)} \cap \Cc$ and all nonzero $v \in \Cb^d$. Next let $\Kb_d \subset \Xb_{d}$ be the subset from Definition~\ref{def:compact_set}. Since $\Kb_d \subset \Xb_{d}$ is compact there exists $C_1 >0 $ such that 
\begin{align*}
\delta_{\Cc^\prime}(0;v) \leq C_1
\end{align*}
for all $\Cc^\prime \in \Kb_d$ and nonzero $v \in \Cb^d$.

We claim that 
\begin{align*}
\delta_{\Cc}(x ;v)\leq \max\left\{ M\delta^{-1/(2+\epsilon)}, C_0C_1\right\} \delta_{\Cc}(x)^{1/(2+\epsilon)}
\end{align*}
for all $x \in \overline{\Bb_d(0;R)} \cap \Cc$ and all nonzero $v \in \Cb^d$.

Fix $x \in \Cc$. If $\delta_{\Cc}(x) \geq \delta$ then 
\begin{align*}
\delta_{\Cc}(x ;v)\leq M \leq M \left(\frac{\delta_{\Cc}(x)}{\delta}\right)^{1/(2+\epsilon)} \leq C\delta_{\Cc}(x)^{1/(2+\epsilon)}
\end{align*}
for all nonzero $v \in \Cb^d$. So suppose that $\delta_{\Cc}(x) < \delta$. Let $\xi \in \partial \Cc$ be a point in $\partial\Cc$ closest to $x$. Then 
\begin{align*}
x = \xi + \delta_{\Cc}(x) \vect{n}(\xi)
\end{align*}
and by construction $\xi \in \Kc$.  

Next we pick points $\xi_1, \xi_2, \dots, \xi_d$ as follows. First let $\xi_1 = \xi$. Next, assuming $\xi_1, \dots, \xi_k$ have been already selected let $P_{k+1}$ be the $(d-k)$-dimensional complex plane through $x$ which is orthogonal to the lines $\overline{x \xi_i}$. Then let $\xi_{k+1}$ be a point in $P_{k+1} \cap \partial \Cc$ which is closest to $x$. By construction $-x+P_2 = T^{\Cb}_{\xi} \partial \Cc$ and hence $(\xi_2-x), \dots, (\xi_d-x) \in T^{\Cb}_{\xi} \partial \Cc$. So by the lemma above 
\begin{align*}
\delta_{\Cc}(x; \xi_i-x) \leq C_0\delta_{\Cc}(x)^{1/(2+\epsilon)}
\end{align*}
for $i \geq 2$. Moreover, since $C_0 \geq 1$ and $\delta_{\Cc}(x) < \delta \leq 1$ we also have 
\begin{align*}
\delta_{\Cc}(x; \xi_1-x) =\delta_{\Cc}(x) \leq C_0\delta_{\Cc}(x)^{1/(2+\epsilon)}.
\end{align*}

Next let $\tau : \Cb^d \rightarrow \Cb^d$ be the affine translation $\tau(z)=z-x$ and let $U$ be the unitary transformation such that 
\begin{align*}
U\tau(\xi_i) = \delta_{\Cc}(x; \xi_i-x) e_i.
\end{align*}
Then let 
\begin{align*}
\Lambda = \begin{pmatrix} 
\delta_{\Cc}(x; \xi_1-x)^{-1} & &  \\
& \ddots & \\
& & \delta_{\Cc}(x; \xi_d-x)^{-1}
\end{pmatrix}.
\end{align*}
Finally let $A$ be the affine map $A = \Lambda U \tau$. Then we have $A\Cc \in \Kb_d$. So if $v \in \Cb^d$ is a unit vector, then
\begin{align*}
\delta_{\Cc}(x;v) = \frac{1}{\norm{\Lambda U v}} \delta_{A\Cc}(0; \Lambda U v) \leq \frac{C_1}{\norm{\Lambda U v}} \leq C_1 C_0\delta_{\Cc}(x)^{1/(2+\epsilon)}
\end{align*}
since 
\begin{align*}
\norm{\Lambda U v} \geq \frac{1}{\norm{\Lambda^{-1}}}\norm{v} \geq \frac{1}{C_0\delta_{\Cc}(x)^{1/(2+\epsilon)}}.
\end{align*}
\end{proof}

\section{Characterizing strong pseudoconvexity}\label{sec:characterizing}

Theorems~\ref{thm:squeezing},~\ref{thm:bergman}, and~\ref{thm:gen_riem} are particular cases of more general theorems which we now describe. In order to state these results we need to define intrinsic functions on the space of convex domains. 

\begin{definition} A function $f:\Xb_{d,0} \rightarrow \Rb$ is called \emph{intrinsic} if  $f(\Cc_1, x_1) = f(\Cc_2, x_2)$ whenever there exists a biholomorphism $\varphi: \Cc_1 \rightarrow \Cc_2$ with $\varphi(x_1) = x_2$. 
\end{definition}

\begin{example} The functions:
\begin{align*}
(\Cc, x) \rightarrow s_{\Cc}(x)
\end{align*}
and 
\begin{align*}
(\Cc, x) \rightarrow \max_{v \in T_x \Cc \setminus \{0\}} \abs{H_{b_{\Cc}}(v) - \frac{-4}{d+1}}
\end{align*}
are both intrinsic.
\end{example}

Since the unit ball is a homogeneous domain we have the following:

 \begin{observation} 
 If $\Bb_d \subset \Cb^d$ is the unit ball and $f:\Xb_{d,0} \rightarrow \Rb$ is an intrinsic function, then $f(\Bb_d, x) = f(\Bb_d, 0)$ for all $x \in \Bb_d$. 
 \end{observation}
 
Recall that the set $\Xb_{d,0}$ has a topology coming from the local Hausdorff topology (see Section~\ref{sec:rescaling} above) and when an intrinsic function is continuous in this topology a version of Klembeck's Theorem (see Theorem~\ref{thm:klembeck} above) holds for convex domains:
 
 \begin{proposition}\cite[Proposition 1.13]{Z2016}
 Suppose $f:\Xb_{d,0} \rightarrow \Rb$ is a continuous intrinsic function and $\Cc$ is a bounded convex domain with $C^2$ boundary. If $\xi \in \partial \Cc$ is a strongly pseudoconvex point of $\partial\Cc$, then 
\begin{align*}
\lim_{z \rightarrow \xi} f(\Cc, z) = f(\Bb_d, 0).
\end{align*}
\end{proposition}

We will prove the following two converses to the above proposition: 

\begin{theorem}\label{thm:cont}(see Subsection~\ref{subsec:thm_cont})
Suppose that $f:\Xb_{d,0} \rightarrow \Rb$ is a continuous intrinsic function with the following property: if $\Cc \in \Xb_d$ and $f(\Cc, x) = f(\Bb_d, 0)$ for all $x \in \Cc$, then $\Cc$ is biholomorphic to $\Bb_d$. 

Then for any $\alpha > 0$ there exists some $\epsilon=\epsilon(d,f,\alpha) > 0$ such that:  if $\Cc \subset \Cb^d$ is a bounded convex domain with $C^{2,\alpha}$ boundary and 
\begin{align*}
\abs{f(\Cc, z) -f(\Bb_d,0)} \leq \epsilon
\end{align*}
outside some compact subset of $\Cc$, then $\Cc$ is strongly pseudoconvex and thus 
\begin{align*}
\lim_{z \rightarrow \partial \Cc} f(\Cc, z) = f(\Bb_d,0).
\end{align*}
\end{theorem}

Some interesting intrinsic functions, for instance the squeezing function, do not appear to be continuous on $\Xb_{d,0}$ but are upper-semicontinuous. So we will also establish the following: 

\begin{theorem}\label{thm:upper_cont}(see Subsection~\ref{subsec:thm_upper_cont})
Suppose that $f:\Xb_{d,0} \rightarrow \Rb$ is an upper semi-continuous intrinsic function with the following property: if $\Cc \in \Xb_d$ and $f(\Cc, x) \geq f(\Bb_d, 0)$ for all $x \in \Cc$, then $\Cc$ is biholomorphic to $\Bb_d$. 

Then for any $\alpha > 0$  there exists some $\epsilon=\epsilon(d,f,\alpha) > 0$ such that: if $\Cc \subset \Cb^d$ is a bounded convex domain with $C^{2,\alpha}$ boundary and 
\begin{align*}
f(\Cc, z) \geq f(\Bb_d,0) - \epsilon
\end{align*}
outside some compact subset of $\Cc$, then $\Cc$ is strongly pseudoconvex.
\end{theorem}

\subsection{An example of Forn{\ae}ss and Wold}\label{subsec:failure} In this subsection we will use an example of Forn{\ae}ss and Wold to show that Theorem~\ref{thm:cont} and Theorem~\ref{thm:upper_cont} both fail for convex domains with $C^2$ boundary. 

\begin{proposition} For any $d \geq 2$ there exists a bounded convex domain $\Cc \subset \Cb^d$ with $C^2$ boundary which is not strongly pseudoconvex, but has the following properties:
\begin{enumerate}
\item If $f_1:\Xb_{d,0} \rightarrow \Rb$ is a continuous intrinsic function, then 
\begin{align*}
\lim_{z \rightarrow \partial \Cc} f_1(\Cc, z) = f_1(\Bb_d,0), 
\end{align*}
\item If $f_2:\Xb_{d,0} \rightarrow \Rb$ is an upper semi-continuous intrinsic function, then 
\begin{align*}
\lim_{z \rightarrow \partial \Cc} f_2(\Cc, z) \geq f_2(\Bb_d,0), 
\end{align*}
\end{enumerate}
\end{proposition}

\begin{proof}
For any $d \geq 2$, Forn{\ae}ss and Wold~\cite{FW2017} have constructed an example of a bounded convex domain $\Cc\subset \Cb^d$ with $C^2$ boundary which is not strongly pseudoconvex, but still satisfies
\begin{align*}
\lim_{z \rightarrow \partial \Cc} s_{\Cc}(z) = 1.
\end{align*}
Now suppose that $f:\Xb_{d,0} \rightarrow \Rb$ is a continuous intrinsic function. We claim that 
\begin{align*}
\lim_{z \rightarrow \partial \Cc} f(\Cc, z) = f(\Bb_d,0), 
\end{align*}
Suppose not then there exist a boundary point $\xi \in \partial \Cc$ and a sequence $z_n \in\Cc$ such that $z_n \rightarrow \xi$ and 
\begin{align*}
\lim_{n \rightarrow \infty} f(\Cc, z_n) \neq f(\Bb_d,0).
\end{align*}
Now by Theorem~\ref{thm:frankel_compactness} we can find affine maps $A_n \in \Aff(\Cb^d)$ such that $A_n(\Cc, z_n)$ converges to some $(\Cc_\infty, z_\infty) \in \Xb_{d,0}$. Since the squeezing function is an upper semi-continuous function on $\Xb_{d,0}$ (see ~\cite[Proposition 7.1]{Z2016}) we have
\begin{align*}
s_{\Cc_\infty}(z_\infty) \geq \limsup_{n \rightarrow \infty} s_{A_n\Cc}(A_nz_n) = \limsup_{n \rightarrow \infty} s_{\Cc}(z_n)=1.
\end{align*}
So $s_{\Cc_\infty}(z_\infty)=1$. Then  $\Cc_\infty$ is biholomorphic to the unit ball by Theorem 2.1 in~\cite{DGZ2012}. Then since $f$ is continuous and intrinsic
\begin{align*}
\lim_{n \rightarrow \infty} f(\Cc, z_n) =\lim_{n \rightarrow \infty} f(A_n\Cc, A_nz_n) = f(\Cc_\infty, z_\infty) = f(\Bb_d,0).
\end{align*}
So we have a contradiction. 

The proof of part (2) is essentially identical.
\end{proof}

\subsection{Rescaling revisited} In this subsection we prove the following rescaling result:

\begin{proposition}\label{prop:recale_2plusAlpha} Suppose $\Cc \subset \Cb^d$ is a convex domain which does not contain any complex lines. If $\Cc$ has $C^{2, \alpha}$ boundary for some $\alpha > 0$ and is not strongly pseudoconvex, then there exists some $\Cc_\infty \in { \rm BlowUp}(\Cc)$ such that: 
\begin{enumerate}
\item $\Cc_\infty \in ie_1+\Kb_d$, 
\item $\Cc_\infty \cap\Span_{\Cb} \left\{ e_2, \dots, e_d \right\} = \emptyset$,  
\item $\Cc_\infty \cap \Cb \cdot e_1 = \{ z e_1 : \Imaginary(z) > 0\}$, and
\item $\delta_{\Cc_\infty}(rie_1; e_2) \leq r^{1/(2+\alpha)}$ for $r \geq 1$.
\end{enumerate}
\end{proposition}

The proof of the Proposition is very similar to the proof of Theorem~\ref{thm:C1}, but we will provide the details anyways. 

\begin{proof}
Since $\Cc$ is not strongly pseudoconvex, there exists a non-strongly pseudoconvex point $\xi \in \partial \Cc$. Then there exist $C, \delta > 0$ and a unit vector $v \in T_{\xi}^{\Cb} \partial \Cc$ such that 
\begin{align*}
\delta_{\Cc}(\xi + r \vect{n}(\xi) ;v) \geq C r^{1/(2+\alpha)}
\end{align*}
for every $r \in (0,\delta]$. Since $\partial \Cc$ is $C^2$, by shrinking $\delta > 0$ if necessary we can assume that 
\begin{align*}
\xi + r\vect{n}(\xi) + r\Db \cdot \vect{n}(\xi) \subset \Cc
\end{align*}
for $r \in (0,\delta]$. 

Then
\begin{align*}
\lim_{ r \rightarrow 0} \frac{r^{1/(2+\alpha)+\epsilon}}{\delta_{\Cc}(\xi + r \vect{n}(\xi) ;v)} = 0
\end{align*}
for every $\epsilon > 0$. 

Now pick $\epsilon_n \rightarrow 0$ and $r_n \rightarrow 0$ such that
\begin{align*}
\lim_{ n  \rightarrow \infty } \frac{r_n^{1/(2+\alpha)+\epsilon_n}}{\delta_{\Cc}(\xi + r_n \vect{n}(\xi) ;v)} = 0.
\end{align*}
Then let $C_n > 0$ be such that 
\begin{align*}
\delta_{\Cc}(\xi + r_n \vect{n}(\xi) ;v) = C_n r_n^{1/(2+\alpha)+\epsilon_n}.
\end{align*}
Since $C_n \rightarrow \infty$, by increasing $r_n$ if necessary we can assume in addition that 
\begin{align*}
\delta_{\Cc}(\xi + r \vect{n}(\xi) ;v) \leq C_n r^{1/(2+\alpha)+\epsilon_n}
\end{align*}
for all $r \in [r_n, \delta]$. Since $\Cc$ contains no complex affine lines, even after possibly increasing each $r_n$ we still have $r_n \rightarrow 0$.

Now let $\tau  \in \Aff(\Cb^d)$ be an affine isometry of $\Cb^d$ such that 
\begin{enumerate}
\item $\tau(\xi)=0$, 
\item $\tau(\xi+ \vect{n}(\xi)) = ie_1$, and
\item $\tau(\xi+v) = e_2$.
\end{enumerate}
Notice that conditions (1) and (2) imply that 
\begin{align*}
T_0 \tau(\Omega) = \{ (z_1, \dots, z_d) \in \Cb^d : \Imaginary(z_1) = 0\}
\end{align*}
and 
\begin{align*}
\tau(\Omega) \subset \{ (z_1, \dots, z_d) \in \Cb^d : \Imaginary(z_1) > 0\}.
\end{align*}

Condition (3) implies that 
\begin{align*}
\delta_{\tau(\Omega)}(r_nie_1;e_2) = C_nr_n^{1/(2+\alpha)+\epsilon_n}.
\end{align*}
Then pick $z_n \in \Cb$ such that $\abs{z_n} = C_nr_n^{1/(2+\alpha)+\epsilon_n}$ and 
\begin{align*}
r_n i e_1 + z_n e_2 \in \partial T\Cc.
\end{align*}
Then consider the diagonal matrix 
 \begin{align*}
 A_n = \begin{pmatrix} \frac{1}{r_n} & & & & \\  & \frac{1}{z_n} & & & \\ & & 1 & & \\ & & & \ddots & \\ & & & & 1 \end{pmatrix}.
 \end{align*}
Let $\Cc_n = A_n \tau(\Cc)$. Since 
\begin{align*}
\xi + r_n\vect{n}(\xi) + r_n\Db \cdot \vect{n}(\xi) \subset \Cc,
\end{align*}
we have 
\begin{align*}
ie_1 + \Db \cdot e_1 \subset \Cc_n.
\end{align*}
Further, by construction: 
 \begin{enumerate}
 \item $\{ (z_1, \dots, z_n) \in \Cb^n : \Imaginary(z_1) = 0\} \cap \Cc_{n} = \emptyset$,
\item $ie_1 + e_2 \notin \Cc_{n}$, and
\item $ie_1 +  \Db \cdot e_2 \subset \Cc_{n}$.
\end{enumerate}
Hence $\Cc_n \cap \Span_{\Cb} \{ e_1, e_2\} \in ie_1+\Kb_2$ where $\Kb_2 \subset \Xb_2$ is the subset from Definition~\ref{def:compact_set}. By Proposition~\ref{prop:two_diml_slice} there exists an affine map $B_n \in \Aff(\Cb^d)$ such that $B_n|_{\Span_{\Cb}\{e_1,e_2\}}= \Id_{\Span_{\Cb}\{e_1,e_2\}}$ and $B_n \Cc_n\in ie_1+ \Kb_d$. 

Now since $\Kb_d$ is compact in $\Xb_d$, we can pass to a subsequence such that $B_n \Cc_n$ converges to some $\Cc_\infty$ in $\Xb_d$. Notice that $B_n \Cc_n = B_n A_n \tau \Cc$ and 
\begin{align*}
ie_1 = (B_nA_n\tau)(\xi + r_n \vect{n}(\xi)).
\end{align*}
 Since $\xi + r_n \vect{n}(\xi)$ converges to the boundary of $\Cc$ and $ie_1 \in \Cc_\infty$ we see that 
\begin{align*}
\Cc_\infty \in \mathrm{BlowUp}(\Cc).
\end{align*}
Moreover, by construction $\Cc_\infty \in \Kb_d$ and $\Cc_\infty \cap\Span_{\Cb} \left\{ e_2, \dots, e_d \right\} = \emptyset$.

As in the proof of Theorem~\ref{thm:C1} for $\eta > 0$ and $r \in (0,\infty]$ let 
\begin{align*}
A(r; \eta) = \{ z \in \Cb : 0 < \abs{z} < r \text{ and } \abs{\Imaginary(z)} < \eta \Real(z) \}.
\end{align*}
Since $\partial \Cc$ is a $C^2$ hypersurface, for any $\eta > 0$ there exists some $r_\eta > 0$ such that 
\begin{align*}
\xi + A(r_\eta; \eta) \cdot \vect{n}(\xi)  \subset \Cc.
\end{align*}
Then for any $\eta > 0$ we have 
\begin{align*}
A(r_\eta/r_n; \eta)\cdot  ie_1 \subset B_nA_n\tau (\Cc)
\end{align*}
and so 
\begin{align*}
A(\infty; \eta)\cdot  ie_1 \subset  \Cc_\infty.
\end{align*}
Since $\eta > 0$ was arbitrary and 
\begin{align*}
\Cc_{\infty} \subset \{ (z_1,\dots,z_d) \in \Cb^d :  \Imaginary(z_1) > 0\}
\end{align*} 
we then have $\Cc_\infty \cap \Cb \cdot e_1 = \{ z e_1 : \Imaginary(z) > 0\}$.

Finally, if $1 \leq r \leq \delta/r_n$, then 
\begin{align*}
\delta_{B_n\Cc_{n}}(rie_1; e_2) & =\delta_{\Cc_{n}}(rie_1; e_2)  = \frac{1}{\abs{z_n}}\delta_{\tau(\Cc)}(r_n r ie_1; e_2)   \\
& = \frac{1}{\abs{z_n}}\delta_{\Cc}(\xi+r_n r\vect{n}(\xi); v)  \leq \frac{1}{\abs{z_n}} C_n (r_nr)^{1/(2+\alpha)+\epsilon_n}= r^{1/(2+\alpha)+\epsilon_n}.
\end{align*}
So for $1 \leq r$ we have 
\begin{equation*}
\delta_{\Cc_{\infty}}(rie_1; e_2)  \leq  r^{1/(2+\alpha)}.
\end{equation*}

\end{proof}

\subsection{The proof of Theorem~\ref{thm:cont}}\label{subsec:thm_cont} Fix $d \geq 2$, a continuous intrinsic function $f: \Xb_{d,0} \rightarrow \Rb$ satisfying the hypothesis of the theorem, and some $\alpha > 0$. Suppose for a contradiction that there exists a sequence of convex domains $\Cc_n \in \Xb_{d,0}$ such that: 
\begin{enumerate}
\item each $\Cc_n$ has $C^{2,\alpha}$ boundary, 
\item each $\Cc_n$ is not strongly pseudoconvex, and
\item for all $n \in \Nb$ 
\begin{align*}
\abs{f(\Cc_n, z) -f(\Bb_d,0)} \leq 1/n
\end{align*}
outside some compact subset of $\Cc_n$.
\end{enumerate}
Now using Proposition~\ref{prop:recale_2plusAlpha} for each $n$ we can find some $\Cc_{n,\infty} \in {\rm BlowUp}(\Cc_n)$ such that 
\begin{enumerate}
\item $\Cc_{n,\infty} \in ie_1 + \Kb_d$, 
\item $\Cc_{n,\infty} \cap\Span_{\Cb} \left\{ e_2, \dots, e_d \right\} = \emptyset$,  
\item $\Cc_{n,\infty} \cap \Cb \cdot e_1 = \{ z e_1 : \Imaginary(z) > 0\}$, and
\item $\delta_{\Cc_{n,\infty}}(e^{r}ie_1; e_2) \leq e^{r/(2+\alpha)}$ for $r \geq 1$.
\end{enumerate}

We claim that 
\begin{align*}
\abs{f(\Cc_{n,\infty}, z) -f(\Bb_d,0)} \leq 1/n
\end{align*}
for all $z \in \Cc_{n,\infty}$. By the definition of ${\rm BlowUp}(\Cc_n)$, there exist a sequence $x_m \in \Cc_n$, a point $x_\infty \in \Cc_{n,\infty}$, and affine maps $A_m \in \Aff(\Cb^d)$ such that $x_m \rightarrow \infty$ in $\Cc_{n}$ and $A_m(\Cc_n, x_m)$ converges to $(\Cc_{n,\infty}, x_\infty)$. Now fix $z \in \Cc_{n,\infty}$ and a relatively compact convex subdomain $\Oc \subset \Cc_{n,\infty}$ which contains $x_\infty$ and $z$. By the definition of the local Hausdorff topology, $\Oc \subset A_m \Cc_n$ for $m$ sufficiently large. So for $m$ sufficiently large $A_m^{-1}(\Oc) \subset \Cc_n$. Then 
\begin{align*}
K_{\Cc_n}(x_m, A_m^{-1} z) \leq K_{A_m^{-1} \Oc}( x_m, A_m^{-1}  z) = K_{\Oc}(A_m x_m, z)
\end{align*}
and since $A_m x_m \rightarrow x_\infty$ we see that 
\begin{align*}
\limsup_{m \rightarrow \infty} K_{\Cc_n}(x_m, A_m^{-1} z) < \infty.
\end{align*}
Since $K_{\Cc_n}$ is a proper metric on $\Cc_n$ and $x_m$ approaches the boundary of $\Cc_n$, we see that $A_m^{-1} z$ approaches the boundary of $\Cc_n$. But then, since $f$ is continuous and intrinsic, 
\begin{align*}
\abs{f(\Cc_{n,\infty}, z) -f(\Bb_d,0)} = \lim_{m \rightarrow \infty} \abs{f(\Cc_n, A_m^{-1}z) -f(\Bb_d,0)}\leq 1/n.
\end{align*}

Now since $\Kb_d \subset \Xb_{d}$ is compact, we can pass to a subsequence such that $\Cc_{n,\infty}$ converges in $\Xb_d$ to some convex domain $\Cc_\infty$. Since $f$ is continuous, we see that 
\begin{align*}
f(\Cc_\infty, z)  = f(\Bb_d,0)
\end{align*}
for all $z \in \Cc_\infty$. So by hypothesis $\Cc_\infty$ is biholomorphic to the unit ball. On the other hand, by the definition of the local Hausdorff topology, we see that 
\begin{enumerate}
\item $ \Cc_\infty \cap\Span_{\Cb} \left\{ e_2, \dots, e_d \right\} = \emptyset$,  
\item $ \Cc_\infty \cap \Cb \cdot e_1 = \{ z e_1 : \Imaginary(z) > 0\}$, and
\item $\delta_{ \Cc_\infty}(e^{r}ie_1; e_2) \leq e^{r/(2+\alpha)}$ for $r \geq 1$.
\end{enumerate}
Hence we have a contradiction with Proposition~\ref{prop:obstruction}. 

\subsection{The proof of Theorem~\ref{thm:upper_cont}}\label{subsec:thm_upper_cont} This is essentially identical to the proof of Theorem~\ref{thm:cont}. 

\subsection{The proof of Theorem~\ref{thm:squeezing}}\label{subsec:squeezing} The function $(\Cc, x) \in \Xb_{d,0} \rightarrow s_{\Cc}(x)$ is an upper semicontinuous intrinsic function (see~\cite[Proposition 7.1]{Z2016}) and by Theorem 2.1 in~\cite{DGZ2012} if $s_\Omega(x) =1 $ for some $x \in \Omega$, then $\Omega$ is biholomorphic to the unit ball. Hence Theorem~\ref{thm:squeezing} follows from Theorem~\ref{thm:upper_cont}. 

\subsection{K{\"a}hler metrics with controlled geometry}\label{subsec:controlled_geom}

We begin by introducing the following class of metrics on a domain which are informally the K{\"a}hler metrics which have controlled geometry relative to the Kobayashi metic. 

\begin{definition}
Suppose $\Omega \subset \Cb^d$ is a bounded domain and $M > 1$. Let $\Gc_M(\Omega)$ be the set of K{\"a}her metrics $g$ on $\Omega$ (with respect to the standard complex structure)
 with the following properties:
\begin{enumerate}
\item $g$ is a $C^2$ metric, 
\item For all $z \in \Omega$ and $v \in \Cb^d$,  
\begin{align*}
\frac{1}{M} \sqrt{g_z(v,v)} \leq k_\Omega(z;v) \leq M \sqrt{g_z(v,v)}.
\end{align*}
\item If $X, v, w \in \Cb^d$, then 
\begin{align*}
\abs{X(g_z(v,w))} \leq M k_\Omega(z;X)k_\Omega(z;v)k_\Omega(z;w).
\end{align*}
\item If $X,Y, v, w \in \Cb^d$, then 
\begin{align*}
\abs{Y(X(g_z(v,w)))} \leq M k_\Omega(z;Y)k_\Omega(z;X)k_\Omega(z;v)k_\Omega(z;w).
\end{align*} 
\item If $X,Y, v, w \in \Cb^d$ and $z_1,z_2 \in \Omega$, then 
\begin{align*}
& \abs{Y(X(g_{z_1}(v,w))) -Y(X(g_{z_2}(v,w)))} \\
& \quad \quad \leq Mk_\Omega(z;Y)k_\Omega(z;X)k_\Omega(z;v)k_\Omega(z;w) K_\Omega(z_1,z_2).
\end{align*} 
\end{enumerate}
\end{definition}

\begin{definition} For $M,d >0$, define a function $h_M:\Xb_{d,0} \rightarrow \Rb$ by letting $h_{M}(\Cc, x)$ be the infimum of all numbers $\epsilon > 0$ such that there exists a metric $g \in \Gc_M(\Cc)$ with
\begin{align*}
\max_{ v,w \in T_z \Cc \setminus \{0\} } \abs{ H_g(v)- H_g(w) } \leq \epsilon \text{ for all } z \in \overline{B}_{\Cc}(x;1/\epsilon)
\end{align*}
where $\overline{B}_{\Cc}(x;r)$ is the closed ball of radius $r$ about the point $x \in \Cc$ with respect to the Kobayashi distance. 
\end{definition}

In~\cite[Proposition 8.2, 8.3]{Z2016} we proved that $-h_M$ is an upper semi-continuous intrinsic function on $\Xb_{d,0}$ and if $h_M(\Cc,x)=0$ for some $x \in \Cc$ then $\Cc$ is biholomorphic to the unit ball in $\Cb^d$. So Theorem~\ref{thm:upper_cont} implies the following:

\begin{corollary} For any $d,M,\alpha >0$ there exists $\epsilon = \epsilon(d,M,\alpha) > 0$ such that: if $\Cc \subset \Cb^d$ is a bounded convex domain with $C^{2,\alpha}$ boundary and 
\begin{align*}
h_M(\Cc, z) \leq \epsilon
\end{align*}
outside some compact subset of $\Cc$, then $\Cc$ is strongly pseudoconvex.
\end{corollary}

Theorem~\ref{thm:gen_riem} is now a simple consequence of this result. 

\begin{proof}[Proof of Theorem~\ref{thm:gen_riem}]
Fix $\epsilon > 0$ with the the following property: if $\Cc \subset \Cb^d$ is a bounded convex domain with $C^{2,\alpha}$ boundary and  
\begin{align*}
h_M(\Cc,z) \leq 2\epsilon
\end{align*}
outside a compact set of $\Cc$, then $\Cc$ is strongly pseudoconvex.

Now suppose that $\Cc \subset \Cb^d$ is a bounded convex domain with $C^{2,\alpha}$ boundary, $K \subset \Cc$ is compact, and there exists a metric $g \in \Gc_M(\Cc)$ such that 
\begin{align*}
\max_{ v,w \in T_z \Cc \setminus \{0\} } \abs{ H_{g}(v)-H_g(w) } \leq \epsilon \text{ for all } z \in \Cc \setminus K.
\end{align*}
We claim that $\Cc$ is strongly pseudoconvex. 

Since $K_{\Cc}$ is a proper distance on $\Cc$ (see Theorem~\ref{thm:barth}), there exists some compact subset $K^\prime \subset \Cc$ such that $B_{\Cc}(x;1/(2\epsilon)) \subset \Cc \setminus K$ for all $x \in \Cc \setminus K^\prime$. Then, with this choice of $K^\prime$, 
\begin{align*}
h_M(\Cc,x) \leq 2\epsilon
\end{align*}
for all $x \in \Cc \setminus K^\prime$. So by our choice of $\epsilon > 0$, $\Cc$ is strongly pseudoconvex. 
\end{proof}

\subsection{The proof of Theorem~\ref{thm:bergman}}\label{subsec:bergman} In~\cite[Proposition 9.1]{Z2016} we proved that for any $d > 0$ there exists some $M_0 =M_0(d) > 1$ such that: if $\Cc \in \Xb_d$ then $b_{\Cc} \in \Gc_M(\Cc)$ for all $M \geq M_0$. So Theorem~\ref{thm:bergman} is a corollary of Theorem~\ref{thm:gen_riem}.

\appendix

\section{Properties of complex hyperbolic space}

In this section we sketch the proof of Theorem~\ref{thm:cplx_hyp}:

\begin{theorem} If $\gamma_1,\gamma_2:\Rb_{\geq 0} \rightarrow \Bb_d$ are geodesic rays such that 
\begin{align*}
\liminf_{s,t \rightarrow \infty} K_{\Bb_d}(\gamma_1(s), \gamma_2(t)) < +\infty, 
\end{align*}
then there exists $T \in \Rb$ such that 
\begin{align*}
\lim_{t \rightarrow \infty} K_{\Bb_d}(\gamma_1(t), \gamma_2(t+T)) =0.
\end{align*}
Moreover, if the images of $\gamma_1$ and  $\gamma_2$ are contained in the same complex geodesic then 
\begin{align*}
\lim_{t \rightarrow \infty} \frac{1}{t} \log K_{\Bb_d}(\gamma_1(t), \gamma_2(t+T)) = -2
\end{align*}
otherwise 
\begin{align*}
\lim_{t \rightarrow \infty} \frac{1}{t} \log K_{\Bb_d}(\gamma_1(t), \gamma_2(t+T)) = -1.
\end{align*}
\end{theorem}

\begin{proof} The first assertion is a consequence of the Kobayashi distance on $\Bb_d$ being induced by a negatively curved Riemannian metric (it is isometric to complex hyperbolic space), see for instance~\cite[Proposition 4.1]{HI1977}.

To establish the second assertion it is easiest to work with the domain 
\begin{align*}
\Pc_d = \left\{ (z_1, \dots, z_d) : \Imaginary(z_1) > \sum_{i=2}^d \abs{z_i}^2\right\}
\end{align*}
which is biholomorphic to $\Bb_d$. 

Suppose that $\gamma_1, \gamma_2 :\Rb_{\geq 0} \rightarrow \Pc_d$ are geodesic rays with 
\begin{align*}
\lim_{t \rightarrow \infty} K_{\Pc_d}(\gamma_1(t), \gamma_2(t)) =0.
\end{align*}
Using the fact that the biholomorphism group $\Aut_0(\Pc_d)$ of $\Pc_d$ acts transitively on the set of geodesic rays in $\Pc_d$, we can assume that 
\begin{align*}
\gamma_1(t) = ie^{2t}e_1.
\end{align*}
Then we must have 
\begin{align*}
\gamma_2(t) = v+\left( \alpha + i(e^{2t} + \norm{v}^2)\right)e_1
\end{align*}
for some $v \in \Span_{\Cb}\{e_2,\dots, e_d\}$ and $\alpha  \in \Rb$. Moreover, $\gamma_1$ and $\gamma_2$ are contained in the same complex geodesic if and only if $v=0$. 

The estimates on 
\begin{align*}
\lim_{t \rightarrow \infty} \frac{1}{t} \log K_{\Pc_d}(\gamma_1(t), \gamma_2(t))
\end{align*}
will follow from the well known fact that if $V \subset \Cb^d$ is an affine subspace which intersects $\Pc_d$ then 
\begin{align*}
K_{V \cap \Pc_d}(z,w) = K_{\Pc_d}(z,w)
\end{align*}
for all $z,w \in V \cap \Pc_d$. 

First suppose that $v=0$. Then 
\begin{align*}
\lim_{t \rightarrow \infty} \frac{1}{t} \log K_{\Pc_d}(\gamma_1(t), \gamma_2(t)) = \lim_{t \rightarrow \infty} \frac{1}{t} \log K_{\Hc}(ie^{2t}, \alpha+ie^{2t})
\end{align*}
where $\Hc = \{ z \in \Cb : \Imaginary(z) > 0\}$. Then 
\begin{align*}
K_{\Hc}(ie^{2t}, \alpha+ie^{2t}) = \frac{1}{2} \arcosh \left( 1 + \frac{\alpha^2}{2e^{4t}} \right)
\end{align*}
and using the fact that $\arcosh(x) = \log (x +\sqrt{x^2-1})$ we then have 
\begin{align*}
K_{\Hc}(ie^{2t}, \alpha+ie^{2t}) = \frac{1}{2} \log \left( 1 + \frac{\alpha^2}{2e^{4t}}  + \frac{\abs{\alpha}}{\sqrt{2}e^{2t}}\right) = \frac{\abs{\alpha}}{2\sqrt{2}}e^{-2t} + {\rm O} \left( e^{-4t} \right).
\end{align*}
So 
\begin{align*}
\lim_{t \rightarrow \infty} \frac{1}{t} \log K_{\Pc_d}(\gamma_1(t), \gamma_2(t)) = -2.
\end{align*}

Next suppose that $v \neq 0$. Then let $\overline{\gamma}_2(t) = v + i(e^{2t}+\norm{v}^2)e_1$. Since 
\begin{align*}
\lim_{t \rightarrow \infty} \frac{1}{t} \log K_{\Pc_d}(\gamma_2(t), \overline{\gamma}_2(t)) = -2
\end{align*}
it is enough to show that 
\begin{align*}
\lim_{t \rightarrow \infty} \frac{1}{t} \log K_{\Pc_d}(\gamma_1(t), \overline{\gamma}_2(t)) = -1.
\end{align*}
Next for $t$ sufficiently large let
\begin{align*}
s_t = t +\frac{1}{2}\log \left( 1 -\frac{\norm{v}^2}{e^{2t}} \right). 
\end{align*}
Then 
\begin{align*}
K_{\Pc_d}(\overline{\gamma}_2(t), \overline{\gamma}_2(s_t)) = \frac{1}{2}\abs{\log \left( 1 -\frac{\norm{v}^2}{e^{2t}} \right)} = \frac{\norm{v}^2}{2}e^{-2t} + {\rm O}\left(e^{-4t} \right)
\end{align*}
so it is enough to show that 
\begin{align*}
\lim_{t \rightarrow \infty} \frac{1}{t} \log K_{\Pc_d}(\gamma_1(t), \overline{\gamma}_2(s_t)) = -1.
\end{align*}
Now since $\overline{\gamma}_2(s_t) = v + ie^{2t}e_1$ and 
\begin{align*}
\Pc_d \cap \left( ie^{2t} + \Cb \cdot v\right) = \left\{ ie^{2t} + z\frac{v}{\norm{v}} : z \in \Cb, \abs{z} \leq e^t \right\}.
\end{align*}
 we have
\begin{align*}
\lim_{t \rightarrow \infty} \frac{1}{t} \log K_{\Pc_d}(\gamma_1(t), \overline{\gamma}_2(s_t)) = \lim_{t \rightarrow \infty} \frac{1}{t} \log K_{e^t \Db}(0, \norm{v} ) = \lim_{t \rightarrow \infty} \frac{1}{t} \log K_{\Db}(0, e^{-t} \norm{v} ) = -1
\end{align*}
where in the last equality we used the fact that  $K_{\Db}(0,z) = \abs{z} + { \rm O}\left(\abs{z}^2\right)$ for $z$ close to $0$. 

\end{proof}

\bibliographystyle{alpha}
\bibliography{complex_kob}

\end{document}